\documentclass[12pt,a4paper]{amsart}%
\usepackage{amsfonts}
\usepackage{amssymb}
\usepackage[utf8]{inputenc}
\usepackage{amsmath}
\usepackage{graphicx}%
\providecommand{\U}[1]{\protect\rule{.1in}{.1in}}
\usepackage[colorlinks=true, linkcolor=red, citecolor=blue]{hyperref}

\newtheorem{theorem}{Theorem}[section]
\newtheorem{proposition}[theorem]{Proposition}
\newtheorem{lemma}[theorem]{Lemma}
\newtheorem{corollary}[theorem]{Corollary}
\newtheorem{remark}{Remark}
\newtheorem*{problem1}{Problem $\mathcal{A}$}
\newtheorem*{problem2}{Problem $\mathcal{B}$}
\newtheorem*{question1}{Question $\mathcal{A}$}

\newtheorem*{question3}{Question $\mathcal{C}$}

\newtheorem*{question5}{Question $\mathcal{E}$}
\newtheorem{claim}{Claim}
\newtheorem*{Claim*}{Claim}
\theoremstyle{remark}

\newtheorem{definition}[theorem]{Definition}
\newcommand{\remove}[1]{ }
\newcommand{\R}{\mathbb{R}}

\newcommand{\C}{\mathbb{C}}
\newcommand{\N}{\mathbb{N}}

\newcommand{\LL}{\mathcal{L}}

\newcommand{\FF}{\mathcal{F}}

\newcommand{\NN}{\mathcal{N}}

\numberwithin{equation}{section}

\begin{document}
\nocite{*}
\title[Higher order Boussinesq system and its properties]{On the well posedness and large-time behavior of higher order Boussinesq system}
\author[Capistrano--Filho]{Roberto A. Capistrano--Filho}
\address{Departamento de Matemática, Universidade Federal de Pernambuco (UFPE), 50740-545, Recife (PE), Brazil.}
\email{capistranofilho@dmat.ufpe.br, roberto.capistranofilho@ufpe.br}
\author[Gallego]{Fernando A. Gallego}
\address{Departamento de Matem\'aticas, Universidad Nacional de Colombia, UNAL, Cra 27 No. 64-60, 170003, Manizales, Colombia}
\email{fagallegor@unal.edu.co, ferangares@gmail.com}
\author[Pazoto]{Ademir F. Pazoto}
\address{Institute of Mathematics, Federal University of Rio de Janeiro (UFRJ)  P.O. Box 68530, CEP 21945-970, Rio de Janeiro (RJ) Brazil.}
\email{ademir@im.ufrj.br}
\date{2018-08-27-a}

\begin{abstract}
A family of Boussinesq systems has been proposed  to describe the bi-directional propagation of small amplitude long waves on the surface of shallow water. In this paper, we investigate the well-posedness and boundary stabilization of the generalized higher order Boussinesq systems of Korteweg-de Vries--type posed on a interval. We design a two-parameter family of feedback laws for which the system is locally well-posed and the solutions of the linearized system are exponentially decreasing in time.
\end{abstract}

\maketitle

\section{Introduction}

\subsection{Presentation of the problem} J. L. Boussinesq introduced in \cite{boussinesq1} several simple nonlinear systems of PDEs,  including the Korteweg-de Vries (KdV) equation, to explain certain physical observations concerning the water waves, e.g. the emergence and stability of solitons. Unfortunately, several systems derived by Boussinesq proved to be ill-posed, so that there was a need to propose other systems similar to Boussinesq's ones but with better mathematical properties. In this spirit, an evolutionary version of the  Boussinesq systems was proposed in \cite[Eqs. (4.7)-(4.8), page 283]{olver1984}:
\begin{equation}\label{fifthb}
\begin{cases}
\eta_t + u_x +\frac16\beta(3\theta^2-1)u_{xxx}+\frac{1}{120}\beta^2(25\theta^4-10\theta^2+1)u_{xxxxx}
\vspace{0.1cm}\\
\hspace{6cm}+\alpha(\eta u)_x+ \frac12 \alpha\beta(\theta^2-1)(\eta u_{xx})_x =0
\vspace{0.2cm}\\
u_t +\eta_x +\beta \left[ \frac12 (1-\theta^2) -\tau \right]\eta_{xxx}+\beta^2\left[\frac{1}{24}(\theta^4-6\theta^2+5)+\frac{\tau}{2}(\theta^2-1)\right] \eta_{xxxxx}
\vspace{0.1cm}\\
\hspace{5cm}\alpha u u_x +\alpha\beta \left[ (\eta\eta_{xx})_x+(2-\theta^2)u_xu_{xx}\right] =0,
\end{cases}
\end{equation}
where $\eta$ and $u$ are real function of the real variables $x$, $t$. The small parameters $\alpha>0$ and $\beta>0$ represent, respectively, the ratio of wave amplitude to undisturbed fluid depth, and the square of the ratio of fluid depth to wave  length, both are assumed to be of the same order of smallness. Finally, $\tau$ represents a dimensionless surface tension coefficient, with $\tau = 0$  corresponding to the case of no surface tension and the velocity potential at height $0\leq \theta \leq 1$. For further discussions on the model and different modelling possibilities, see, e.g. \cite{bona2002,bona2004,bona1976,daripa2003,olver1984,SauWanXu2017}.

The goal of this paper is to investigate two problems that appear on the mathematical theory when we consider the study of PDEs. The first one is the global well-posedness, in time, of system \eqref{fifthb}, which is so-called \textit{fifth order KdV--type system}. Another problem is concerned with boundary stabilization of the linearized system associated to (\ref{fifthb}).

First, we consider the following system, carefully derived by \eqref{fifthb} in a short Appendix at the end of this paper,
\begin{equation}
\begin{cases}\label{int1}
\eta_t + u_x-au_{xxx}+ a_1 (\eta u)_x +a_2(\eta u_{xx})_x + b u_{xxxxx}= 0,&  \text{in } (0,L)\times(0,\infty), \\
u_t +\eta_x -a\eta_{xxx} +a_1uu_x+a_3(\eta\eta_{xx})_{x}+  a_4u_xu_{xx}    +b\eta_{xxxxx}=0,&  \text{in } (0,L)\times(0,\infty), \\
\eta(x,0)= \eta_0(x), \quad u(x,0)=  u_0(x),&  \text{in } (0,L), \\
\end{cases}
\end{equation}
where $a>0$, $b>0$, $a\neq b$, $a_1>0$, $a_2<0$, $a_3>0$ and $a_4>0$, with the following boundary conditions
\begin{equation}
\begin{cases}\label{int2}
\eta(0,t)=\eta(L,t)=\eta_{x}(0,t)=\eta_x(L,t)=0, & \text{in} \,\,  (0,\infty), \\
u(0,t)=u(L,t)=u_{x}(0,t)=u_x(L,t)=0, & \text{in} \,\,  (0,\infty), \\
u_{xx}(0,t)+\alpha_1\eta_{xx}(0,t)=0, \,\,\,\, u_{xx}(L,t)-\alpha_2\eta_{xx}(L,t)=0, & \text{in} \,\,  (0,\infty),
\end{cases}
\end{equation}
for $\alpha_1,\alpha_2\in \mathbb{R}^{+}_{\ast}$.

The energy associated to the model is given by
\begin{equation}\label{energy}
E(t):=\frac{1}{2}\int_0^L(\eta^2(x,t)+u^2(x,t))dx,
\end{equation}
and, at least formally, we can verify that $E(t)$ satisfies
\begin{multline}\label{dis}
\frac{d}{dt}E(t)=-\alpha_1b|\eta_{xx}(0,t)|^2 -\alpha_2b|\eta_{xx}(L,t)|^2 -\frac{a_1}{2}\int_0^L\eta^2 u_x dx \\
-\frac{a_2}{2}\int_0^L \eta^2 u_{xxx} dx +a_3\int_0^L\eta\eta_{xx}u_xdx +\frac{a_4}{2}\int_0^Lu_x^3dx.
\end{multline}
Indeed, if we multiply the first equation of (\ref{int1}) by $\eta$, the second one by $u$ and integrate by parts over $(0, L)$, we obtain \eqref{dis}, by using the boundary conditions \eqref{int2}. This indicates that $E(t)$ does not have a definite sign, but the boundary conditions play the role of a feedback damping mechanism for the linearized system, namely,
\begin{equation}\label{b1.a}
\begin{cases}
\eta_t + u_x-au_{xxx}+ b u_{xxxxx}= 0, &  \text{in } (0,L)\times(0,\infty), \\
u_t +\eta_x -a\eta_{xxx} +b\eta_{xxxxx}=0,&  \text{in }   (0,L)\times(0,\infty), \\
\eta(x,0)= \eta_0(x), \quad u(x,0)=  u_0(x),&  \text{in }   (0,L),
\end{cases}
\end{equation}
with the boundary conditions given by \eqref{int2}.

Then, the following questions arise:
\begin{problem1}
Does $E(t)\to 0$ as $t\to\infty$?  If it is the case, can we find a decay rate of $E(t)$?
\end{problem1}
The problem might be easy to solve when the underlying model has a intrinsic dissipative nature. Moreover, in the context of coupled systems, in order to achieve the desired decay property, the damping mechanism has to be designed in an appropriate
way to capture all the components of the system.

Before presenting an answer for Problem $\mathcal{A}$, it is necessary to investigate the global well-posedeness of the full system \eqref{int1}-\eqref{int2}. Thus, the following issue appears naturally:
\begin{problem2}
Is the fifth order KdV--type system globally well-posed in time, with initial data in $H^s(0,L)$, for some $s\in \mathbb{R}^{+}$?
\end{problem2}

\subsection{Some previous results}  It is by now well know that mathematicians are interested in the well-posedness of dispersive equation which depends on smoothing effects associated to datum (initial value or boundary value).  The well-posedness of the initial value problem for single KdV equation and single fifth order KdV equation was deeply investigated. For an extensive reading on the subject see, for instance,  \cite{BonaSunZhang2003,CoSaut1988,glassguerrero,KenPonVeg1991,ZhaoZhang2014} and the reference therein.  In contrast, the well-posedness theory for the coupled system of KdV--type is considerably less advanced than the theory for single KdV--type equations \cite{bona2002,bona2004,SauXu2012,SauXu2012a,SauWanXu2017}. The same is true for the stabilization properties.

In additional, other interesting problem, as mentioned previously, in the stabilization problem. Problem $\mathcal{A}$ was first addressed in \cite{pazoto2008} for a Boussinesq system of KdV-KdV type
\begin{equation}
\left\{
\begin{array}
[c]{lll}%
\eta_{t}+u_{x}+ (\eta u)_x + u_{xxx}=0 &  & \text{in }(
0,L)  \times(   0,T)  \text{,}\\
u_{t}+\eta_{x} + uu_x + \eta_{xxx}=0 &  & \text{in }(   0,L)
\times(   0,T)  \text{,}\\
\eta(x,0)  =\eta_{0}(   x)  \text{, \ } u(x,0)  =u_{0}(   x)  &  & \text{in }(   0,L),
\end{array}
\right.  \label{int_34e}%
\end{equation}
with the boundary conditions%
\begin{equation}
\left\{
\begin{array}
[c]{lll}%
u(0,t)  =u_{xx}(0,t)  =0, &  & \text{in
}(   0,T)  \text{,}\\
u_{x}(0,t)  =\alpha_{0}\eta_{x}(0,t), \quad u_{x}(L,t)  =-\alpha_{1}\eta_{x}(L,t) &  & \text{in }( 0,T)  \text{,}\\
u(L,t)  =\alpha_{2}\eta(L,t),   \quad u_{xx}(L,t)  =-\alpha_{2}\eta_{xx}(L,t)  &  &
\text{in }(   0,T),
\end{array}
\right.  \label{int_32e}%
\end{equation}
where $\alpha_{0}\geq 0$, $\alpha_{1}>0$ and $\alpha _2>0$. Note that, with boundary conditions (\ref{int_32e}), we have the following identity $$\frac{d}{dt}E(t)=-\alpha_2|\eta(L,t)|^2-\alpha_1|\eta_x(L,t)|^2-\alpha_0|\eta_x(0,t)|^2-\frac{1}{3}u^3(L,t)-\int_0^L(\eta u)_x\eta dx,$$ which does not have a definite sign. In this case, first the authors studied the linearized system to derive some a priori estimates and the exponential decay in the $L^2$--norm. It is established the Kato smoothing effect by means of the multiplier method, while the exponential decay is obtained with the aid of some compactness arguments that reduce the issue to prove a unique continuation property for a spectral problem associated to the space operator (see,
for instance, \cite{bardos,Rosier}). The exponential decay estimate of the linear system is then combined with the contraction mapping theorem in a convenient weighted space to prove the global well-posedness together with the exponential stability property of the nonlinear system \eqref{int_34e}-\eqref{int_32e} with small data.

Recently, in \cite{CaPaRo}, the authors studied a similar  boundary stabilization problem for the system (\ref{int_34e}) with less amount of damping. More precisely, the following boundary conditions was considered
\begin{equation*}
\left\{
\begin{array}
[c]{lll}%
\eta(0,t)  =0, \ \eta (L,t)=0, \ \eta _x(0,t)=0\ , &  & \text{in } (0,T)  \text{,}\\
u (0,t)  =0 ,\  u(L,t) =0 ,\  u_{x}(L,t)  =-\alpha \eta _x (L,t)  &  & \text{in }( 0,T)  \text{,}\\
\end{array}
\right.
\end{equation*}
with $\alpha >0$. In this case, it follows that
\begin{equation*}
\frac{d}{dt}E(t)=- \alpha |\eta_x(L,t)|^2 -\frac{1}{3}u^3(L,t)-\int_0^L(\eta u)_x\eta dx.
\end{equation*}
Proceeding as in \cite{pazoto2008} the local exponential decay is also obtained for solution issuing from small data. However, due to the lack of dissipation, the unique continuation issue for the linearized system can not be obtained by standard methods. In order to overcome this difficult, the spectral problem was then solved by extending the function $(\eta,u)$ by 0 outside $(0, L)$, by taking its Fourier transform and by using Paley-Wiener theorem. Finally, the problem was reduced to check for which values of $L>0$ two functions are entire for a set of parameters. Then, the authors concluded that the stabilization properties holds if and only if the length $L$ does not belong to the following critical set \[\mathcal{N} := \{\pi \sqrt{\frac{k^2+kl+l^2}{3}}; \ k,l\in \N ^*\}.\]
We point out that the same set was obtaned by Rosier, \cite{Rosier} while studying the boundary controllability of the KdV equation with a single control in $L^2(0,T)$ acting on the Neumann boundary condition. This shows that the linearized Boussinesq system inherits some interesting properties initially observed for the KdV equation.

\subsection{Main results and comments} In the present work, we address the problems described in the previous subsection and our main results provide a partial positive answer for the Problems $\mathcal{A}$ and $\mathcal{B}$.  In order to give an answer for Problem $\mathcal{B}$, we apply the ideas suggested in \cite{CaGa,CaPaRo}, therefore, let us consider
$$
\overline{X}_3 = \{ (\eta, u) \in [ H^3(0,L)\cap H^2_0(0,L) ] ^2; \
 \eta _{xx}(0)=v_{xx}(L)=0\}.
$$
With this notation, one of the main result of this article can be read as follows:
\begin{theorem}\label{main_int1}
Let $T>0$. Then, there exists $\rho=\rho(T)>0$ such that,  for every $(\eta_0,u_0) \in \overline{X}_3$ satisfying
\begin{equation*}
\|(\eta_0,u_0)\|_{ \overline{X}_3 }< \rho,
\end{equation*}
 there exists a unique solution $(\eta, u) \in C([0,T]; \overline{X}_3)$ of \eqref{int1}-\eqref{int2}. Moreover
\begin{align*}
\|(\eta,u)\|_{C([0,T]; \overline{X}_3)} \leq C \|(\eta_0,u_0)\|_{ \overline{X}_3}
\end{align*}
for some positive constant $C=C(T)$.
\end{theorem}

In order to prove Theorem \ref{main_int1} we first analyze the linearized model by using a semigroup approach. Moreover, by using multiplier techniques, we also obtain the so-called Kato smoothing effect, which is crucial to study the stabilization problem. In what concerns the full system, the idea is to combine the linear theory and a fixed point argument. However, the linear theory described above seems to be unable to provide the a priori bounds needed to use a fixed point argument. To overcome this difficult, we consider solutions obtained \textit{via} transposition method, which leads to consider a duality argument and the solutions of the corresponding adjoint system. Then, the existence and uniqueness can be proved by using the Riesz-representation theorem that gives, at first, a solution which is not continuous in time, only $L^\infty$. The continuity is then obtained with the aid of what is known as \textit{hidden regularity} of the boundary terms of the adjoint system. In fact, we prove that such system has a class of solutions which belong to appropriate spaces possessing boundary regularity. On the other hand, it is also important to note that identity \eqref{dis} does not provide any global (in time) a priori bounds for the solutions. Consequently, it does not lead to the existence of a global (in time) solution in the energy space. The same lack of a priori bounds occurs when higher order Sobolev norms are considered (e.g. $H^s$-norm).

With the damping mechanism proposed in (\ref{int2}), the stabilization of the linearized higher order Boussinesq system  (\ref{b1.a}) holds for \textit{any length} of the domain. Thus, the second main result of this paper is the following:
\begin{theorem}\label{main_int}
Assume that $\alpha_1>0$, $\alpha_2 > 0$ and $L>0$. Then, there exist some constants $C_0, \mu_0 > 0$, such that, for any $(\eta_0,u_0) \in X_0:=[L^{2}\left(  0,L\right)]^2$, system (\ref{b1.a})-(\ref{int2}) admits a unique solution $$(\eta,u)\in C^{0}\left(  \left[  0,T\right]  ;X_0\right)  \cap
L^{2}\left(  0,T;[H^{2}\left(  0,L\right)]^2  \right)$$ satisfying
\begin{equation*}
\|(\eta(t),u(t))\|_{X_0}\leq C_0e^{-\mu_0 t}\|(\eta_0,u_0)\|_{X_0}, \quad \forall t \geq 0.
\end{equation*}
\end{theorem}

In order to prove Theorem \ref{main_int} we proceed as in \cite{CaPaRo,pazoto2008}, i.e, combining multipliers and compactness arguments which reduces the problem to show a unique continuation result for the state operator. To prove this result, we extend the solution under consideration by zero in $\mathbb{R}\setminus[0,L]$ and take the Fourier transform. However, due to the complexity of the system, after taking the Fourier transform of the extended solution $(\eta,u)$  it is not possible to use the same techniques used in \cite{CaPaRo}.
Thus, to prove our main result we proceed as Santos \textit{et al.} \cite{santos}.

For a better understanding we will introduce a general framework to explain the idea of the proof. After to take Fourier transform, the issue is to establish when a certain quotient of entire functions still turn out to be an entire function. We then pick a polynomial $q:\mathbb{C}\to\mathbb{C}$ and a family of functions $$N_{\alpha}:\mathbb{C}\times(0,\infty)\to\mathbb{C},$$
with $\alpha\in\mathbb{C}^4\setminus \{0\}$, whose restriction $N_{\alpha}(\cdot,L)$ is entire for each $L>0$. Next, we consider a family of functions $f_{\alpha}(\cdot,L)$, defined by $$f_{\alpha}(\mu,L)=\frac{N_{\alpha}(\mu,L)}{q(\mu)},$$ in its maximal domain. The problem is then reduced to determine $L>0$ for which there exists $\alpha\in\mathbb{C}^4\setminus \{0\}$ such that $f_{\alpha}(\cdot,L)$ is entire. 
In contrast with the analysis developed in \cite{CaPaRo}, this approach does not provide us an explicit characterization of a critical set, if it exists, only ensure that the roots of $f$ have a relations with the M\"obius transform (see the proof of Theorem \ref{main_int} above).


The remaining part of this paper is organized as follows: In Section \ref{Sec1}, we establish the well-posedness of the linearized system. We also derived a series of linear estimates for a conservative linear Boussinesq system which will we used to prove the well-posedness for the full system (\ref{int1})-(\ref{int2}).  Section \ref{Sec2}, is then devoted to prove the well-posedness for the nonlinear system. In section \ref{Sec3}, we prove an \textit{observability inequality} associated to (\ref{b1.a})-(\ref{int2}), which plays a crucial role to get second result of this paper, Theorem \ref{main_int}, proved in the same section.  Finally, some additional comments and open problems are proposed in the Section \ref{Sec4}. We also include an Appendix with a detailed derivation of the system (\ref{int1}).

\section{Well-posedness: Linear system}\label{Sec1}
The goal of the section is to prove the well-posedness of the linearized system. In order to do that, we use the semigroup theory and multiplier techniques, which allow us to derived so-called Kato smoothing effect. We also use the same approach to study a similar conservative linear Boussinesq system that will be used to study the full system (see, Definition \ref{transposition}).

\subsection{Well-posedness: linear system}
We will study the existence of solutions of the linear homogeneous system associated to \eqref{b1.a}, namely
\begin{equation}\label{homo1}
\begin{cases}
\eta_t + u_x-au_{xxx}+b u_{xxxxx}= 0, & \text{in} \,\, (0,L)\times (0,T),\\
u_t +\eta_x -a\eta_{xxx}+b \eta_{xxxxx} =0,  & \text{in} \,\, (0,L)\times (0,T), \\
\eta(0,t)=\eta(L,t)=\eta_{x}(0,t)=\eta_x(L,t)=0,  & \text{in} \,\,  (0,T), \\
u(0,t)=u(L,t)=u_{x}(0,t)=u_x(L,t)=0,  & \text{in} \,\,  (0,T),  \\
u_{xx}(0,t)+\alpha_1\eta_{xx}(0,t)=0, & \text{in} \,\, (0,T), \\
u_{xx}(L,t)-\alpha_2\eta_{xx}(L,t)=0,  & \text{in} \,\, (0,T), \\
\eta(x,0)= \eta_0(x), \quad u(x,0)=  u_0(x), & \text{in} \,\,  (0,L).
\end{cases}
\end{equation}

We consider $X_0$ with the usual inner product and  the operator $A:D(A)\subset X_0 \to X_0$ with domain
\begin{equation*}
D(A)=\{(\eta,u)\in [H^5(0,L)\cap H^2_0(0,L)]^2: u_{xx}(0)+\alpha_1\eta_{xx}(0)=0, u_{xx}(L)-\alpha_2\eta_{xx}(L)=0\},
\end{equation*}
defined by
\begin{equation*}
A(\eta,u)=(-u_x+au_{xxx}-bu_{xxxxx},-\eta_x+a\eta_{xxx}-b\eta_{xxxxx}).
\end{equation*}
Let us denote $X_5= D(A)$. Moreover, we introduce the Hilbert space $$X_{5\theta} := [X_0,X_5]_{[\theta]}, \quad \text{for $0<\theta <5$}, $$ where $[X_0,X_5]_{[\theta]}$ denote the the Banach space obtained by the complex interpolation method (see, e.g., \cite{bergh}).

Then, the following result holds:
\begin{proposition}\label{semigroup}
If $\alpha_i\geq 0$, $i=1,2$, then $A$ generates a $C_0$-semigroup of contraction $(S(t))_{t\geq 0}$ in $X_0$.
\end{proposition}
\begin{proof}
Clearly, $A$ is densely defined and closed, so we are done if we prove that $A$ and its adjoint $A^*$ are both dissipative in $X_0$. It is easy to see that $$A^*:D(A^*)\subset X_0 \longrightarrow X_0$$ is given by $A^*(\varphi,\psi)=(\psi_x-a\psi_{xxx}+b\psi_{xxxxx},\varphi_x-a\varphi_{xxx}+b\varphi_{xxxxx})$ with domain
\begin{align*}
D(A^*)=\{(\varphi,\psi)\in X_5&: \varphi(0)=\varphi(L)=\varphi_{x}(0)=\varphi_x(L)=0, \\
&\ \ \psi(0)=\psi(L)=\psi_{x}(0)=\psi_x(L)=0, \\
&\ \ \psi_{xx}(0)-\alpha_1\varphi_{xx}(0)=0,  \psi_{xx}(L)+\alpha_2\varphi_{xx}(L)=0\}.
\end{align*}
Pick any $(\eta,u) \in D(A)$. Multiplying the first equation of \eqref{homo1} by $\eta$, the second one by $u$ and integrating by parts, we obtain
\begin{align*}
\left( A(\eta,u),(\eta,u)\right)_{X_0}=-\alpha_2 b\eta_{xx}^2(L)-\alpha_1b\eta_{xx}^2(0) \leq 0,
\end{align*}
which demonstrates that $A$ is a dissipative operator in $X_0$. Analogously, we can deduce that, for any $(\varphi,\psi) \in D(A^*)$,
\begin{align*}
\left( A^*(\varphi,\psi),(\varphi,\psi)\right)_{X_0}=-\alpha_2 b\varphi_{xx}^2(L)-\alpha_1b\varphi^2_{xx}(0) \leq 0,
\end{align*}
so that $A^*$ is dissipative, as well. Thus, the proof is complete.
\end{proof}

As a direct consequence of Proposition \ref{semigroup} and the general theory of evolution equation, we have the following existence and uniqueness result:

\begin{proposition}\label{existence}
Let $(\eta_0,u_0)\in X_0$. There exists a unique mild solution $(\eta,u)=S(\cdot)(\eta_0,u_0)$ of \eqref{homo1} such that $(\eta,u)\in C([0,T];X_0)$. Moreover, if $(\eta_0,u_0)\in D(A)$, then \eqref{homo1} has a unique (classical) solution $(\eta,u)$ such that $$(\eta,u)\in C([0,T];D(A))\cap C^1(0,T;X_0).$$
\end{proposition}

The following proposition provides useful estimates for the standard energy and the Kato smoothing effect for the mild solutions of \eqref{homo1}.
\begin{proposition}\label{prop2}
Let $(\eta_0,u_0) \in X_0$ and $(\eta(t),u(t))=S(t)(\eta_0,u_0)$. Then, for any $T>0$, we have that
\begin{multline}\label{energia1}
\quad \|(\eta_0(x),u_0(x))\|^2_{X_0}-\|(\eta(x,T),u(x,T))\|^2_{X_0}\\
= \int_0^T\left(\alpha_2 b|\eta_{xx}(L,t)|^2+\alpha_1b|\eta_{xx}(0,t)|^2\right) dt \qquad \qquad \qquad 
\end{multline}
and
\begin{equation}\label{energia2}
\begin{split}
\frac{T}{2}\|(\eta_0(x),u_0(x))\|^2_{X_0}&=\frac12\|(\eta(x,t),u(x,t))\|^2_{L^2(0,T;X_0)}\\
&+ \alpha_2 b\int_0^T(T-t)|\eta_{xx}(L,t)|^2dt +\alpha_1b\int_0^T(T-t)|\eta_{xx}(0,t)|^2dt.
\end{split}
\end{equation}
Furthermore, $(\eta,u) \in L^2(0,T;X_2)$ and
\begin{equation}\label{kato1}
\|(\eta,u)\|_{L^2(0,T;X_2)}\leq C\|(\eta_0,u_0)\|_{X_0},
\end{equation}
where $C=C(a,b,T)$ is a positive constant.
\end{proposition}
\begin{proof}
We obtain the estimates \eqref{energia1}-\eqref{kato1} using multiplier techniques. Pick any $(\eta_0,u_0) \in D(A)$. Multiplying the first equation in \eqref{homo1} by $\eta$, the second one by $u$, adding the resulting equations and integrating over $(0, L)\times (0, T)$, we obtain  \eqref{energia1} after some integration by parts.
The identity may be extended to any initial state $(\eta_0,u_0) \in X_0$ by a density argument. Moreover, multiplying the first equation in \eqref{homo1} by $(T - t)\eta$,  the second by $(T - t)u$ and integrating over $(0, L) \times  (0, T)$ we derive \eqref{energia2} in a similar way.

Let us proceed to the proof of \eqref{kato1}. Multiply the first equation by $x u$, the second one by $x \eta$ and integrate over $(0, L)\times(0, T )$. Adding the obtained equations we get that
\begin{multline*}
\int_0^T\int_0^L x (\eta u)_tdxdt + \int_0^T\int_0^L \frac{x}{2}(|\eta|^2+|u|^2)_xdxdt  \\
-a\int_0^T\int_0^L x (\eta\eta_{xxx}+uu_{xxx})dxdt  +b\int_0^T\int_0^L x (\eta\eta_{xxxxx}+uu_{xxxxx})dxdt =0 .
\end{multline*}
After some integration by parts, it follows that
\begin{multline*}
\int_0^T\int_0^L x (\eta u)_tdxdt - \frac12\int_0^T\int_0^L (|\eta|^2+|u|^2)dxdt
  -\frac{3a}{2}\int_0^T\int_0^L  (|\eta_{x}|^2+|u_{x}|^2)dxdt \\ -\frac{5b}{2}\int_0^T\int_0^L  (|\eta_{xx}|^2+|u_{xx}|^2)dxdt
  +\frac{bL}{2}\int_0^T(|\eta_{xx}(L,t)|^2+|u_{xx}(L,t)|^2)dt =0,
\end{multline*}
hence,
\begin{equation}\label{e1}
\begin{split}
 \frac12\int_0^T\int_0^L (|\eta|^2+|u|^2)dxdt  &+\frac{3a}{2}\int_0^T\int_0^L  (|\eta_{x}|^2+|u_{x}|^2)dxdt \\& +\frac{5b}{2}\int_0^T\int_0^L  (|\eta_{xx}|^2+|u_{xx}|^2)dxdt\\&\leq L \int_0^L (\eta(x,T) u(x,T)-\eta_0(x)u_0(x))dx \\&+\frac{bL(1+\alpha_2^2)}{2}\int_0^T|\eta_{xx}(L,t)|^2dt.
\end{split}
\end{equation}
By using \eqref{energia1} and Young inequality in the first integral of the right hand side in \eqref{e1}, we have that
\begin{equation*}
\begin{split}
 \frac12\int_0^T\int_0^L (|\eta|^2+|u|^2)dxdt &  +\frac{3a}{2}\int_0^T\int_0^L  (|\eta_{x}|^2+|u_{x}|^2)dxdt +\frac{5b}{2}\int_0^T\int_0^L  (|\eta_{xx}|^2+|u_{xx}|^2)dxdt\\ \leq& \frac{L}{2}\int_0^L (|\eta(x,T)|^2 +|u(x,T)|^2)dx\\& +\frac{L}{2}\left( 1+\frac{1+\alpha_2^2}{\alpha_2}\right)\int_0^L(|\eta_0(x)|^2+|u_0(x)|d^2)dx.
 \end{split}
\end{equation*}
Clearly, \eqref{energia1} implies that $E(T)\leq E(0)$, thus
\begin{multline*}
 \int_0^T\int_0^L (|\eta|^2+|u|^2)dxdt  +3a\int_0^T\int_0^L  (|\eta_{x}|^2+|u_{x}|^2)dxdt  +5b\int_0^T\int_0^L  (|\eta_{xx}|^2+|u_{xx}|^2)dxdt \\
\leq L\left( 1+\frac{1+\alpha_2^2}{\alpha_2}\right)\int_0^L(|\eta_0(x)|^2+|u_0(x)|d^2)dx.
\end{multline*}
Then, \eqref{kato1} holds.
\end{proof}

\subsection{Well-posedness: a conservative linear system}
This subsection is devoted to analyze a conservative linear model that will be used to derived the nonlinear theory.

Let us starting by introducing the spaces
$$
X_0:=\overline{X}_0:= L^2(0,L)\times L^2(0,L),
$$
\begin{equation}\label{newA1}
\overline{X}_5=\{(\varphi,\psi)\in [H^5(0,L)\cap H^2_0(0,L)]^2: \varphi_{xx}(0)=\psi_{xx}(L)=0 \},
\end{equation}
and
$$
\overline{X}_{5\theta }:= [\overline{X}_0,\overline{X}_5]_{ [\theta ]}, \quad \textrm{ for } 0<\theta <1,
$$
where $[\overline{X}_0,\overline{X}_5]_{[\theta ]}$ denote the the Banach space obtained by the  complex  interpolation method
(see, e.g., \cite{bergh}).It is easily seen that
\begin{align*}
\overline{X}_1 =& H^1_0(0,L)\times H^1_0(0,L), \\
\overline{X}_2 =& \{ (\eta, v) \in [ H^2(0,L)\cap H^1_0(0,L) ] ^2;  \eta _{x}(L)=v_{x}(0)=0\}. \\
\overline{X}_3 =& \{ (\eta, v) \in [ H^3(0,L)\cap H^2_0(0,L) ] ^2; \
 \eta _{xx}(0)=v_{xx}(L)=0\}.
 \end{align*}
On the other hand, we shall use at some place below the following space
\begin{eqnarray*}
\overline{X}_7 :=  \{ (\eta , v) \in [H^7(0,L)\cap H^2_0(0,L)]^2;\   \eta _{xx}(0)=v_{xx}(L)=0,\\
\qquad\qquad \qquad \qquad\qquad \qquad \qquad \ \ \ \  -a\eta _{3x } (L) + b\eta _{5x} (L) =-av_{3x} (L)+bv_{5x}(L)=0,\\
\qquad\qquad \qquad \qquad\qquad \qquad \qquad \ \  \ \ \ \  -a\eta _{3x } (0) + b\eta _{5x} (0) =-av_{3x} (0)+bv_{5x}(0)=0,\\
\qquad\qquad \qquad \qquad\qquad \qquad \qquad \ \ \ \  -a\eta _{4x } (0) + b\eta _{6x} (0) =-av_{4x} (L)+bv_{6x}(L)=0\},
\end{eqnarray*}
endowed with its natural norm. The space 
$$\overline{X}_{-s}=(\overline{X}_s)'$$ 
denotes the dual of $\overline{X}_s$ with respect to the pivot space $\overline{X}_0=L^2(0,L)\times L^2(0,L)$. 
The bracket $\langle .,. \rangle _{\overline{X}_{-s},\overline{X}_s}$ stands for the duality between $\overline{X}_{-s}$ and $\overline{X}_s$ . 

Now, we turn our attention to the well-posedness of the system associated to the differential operator $\widetilde{A}$, given by
\begin{equation}\label{newA}
\widetilde{A}(\varphi,\psi)=(-\psi_x+a\psi_{xxx}-b\psi_{xxxxx},-\varphi_x+a\varphi_{xxx}-b\varphi_{xxxxx}),
\end{equation}
with domain, $D(\widetilde{A})=\overline{X}_5$. More precisely, we consider the following system 
\begin{equation}
\label{new1}
\begin{cases}
\eta_t + u_x-au_{xxx}+bu_{xxxxx}= 0, & \text{in} \,\, (0,L)\times (0,T),\\
u_t +\eta_x -a\eta_{xxx} +b\eta_{xxxxx}=0,  & \text{in} \,\, (0,L)\times (0,T), \\
\eta(0,t)=\eta(L,t)=\eta_x(0,t)=\eta_x(L,t)=\eta_{xx}(0,t)=0,& \text{on} \,\, (0,T), \\
u(0,t)=u(L,t)=u_x(0,t)=u_x(L,t)=u_{xx}(L,t)=0, & \text{on} \,\, (0,T), \\
\eta(x,0)= \eta_0(x), \quad u(x,0)=  u_0(x), & \text{on} \,\, (0,L).
\end{cases}
\end{equation}

\begin{proposition}\label{propA}
The operator $\widetilde{A}$ is skew-adjoint in $\overline{X}_0$, and thus it generates a group of isometries $(e^{tA})_{t\in \R}$ in $\overline{X}_0$.
\end{proposition}
\begin{proof}
We show that $\widetilde{A}^{*}=-\widetilde{A}$. First, we prove that $-\widetilde{A}\subset \widetilde{A}^{*}$. Indeed, for any $(\eta, u), (\theta,v) \in D(\widetilde{A})$, we have after some integration by parts,
\begin{equation*}
\begin{split}
\left( (\theta,v), \widetilde{A}(\eta,u)\right)_{\overline{X}_0} =& -\int_0^L\left[ \theta (u_x-au_{xxx}+bu_{xxxxx})+v(\eta_x-a\eta_{xxx}+b\eta_{xxxxx})\right]dx \\
= &\int_0^L\left[ u (\theta_x-a\theta_{xxx}+b\theta_{xxxxx})+\eta(v_x-av_{xxx}+bv_{xxxxx})\right]dx \\
=& \left( \widetilde{A}(\theta,v), (\eta,u)\right)_{\overline{X}_0}.
\end{split}
\end{equation*}
Now, we prove that $\widetilde{A}^{*}\subset -\widetilde{A}$. Pick any $(\theta, v) \in D(\widetilde{A}^{*})$. Then, for some positive constant $C$, we have that
\begin{align*}
\left | \left( (\theta,v), \widetilde{A}(\eta,u)\right)_{\overline{X}_0} \right| \leq C \|(\eta,u)\|_{\overline{X}_0}, \quad \forall (\eta,u) \in D(A).
\end{align*}
Thus, it follows that
\begin{equation}\label{a2}
\begin{split}
\left|\int_0^L\left[ \theta (u_x-au_{xxx}+bu_{xxxxx})+v(\eta_x-a\eta_{xxx}+b\eta_{xxxxx})\right]dx\right| \\ \leq C \left( \int_0^L (\eta^2+u^2)dx\right)^{\frac12}, \quad \forall (\eta,u) \in D(A).
\end{split}
\end{equation}
Taking $\eta \in C^{\infty}_c(0,L)$ and $u=0$, we deduce from \eqref{a2} that $v \in H^5(0,L)$. Similarly, we obtain that $\theta \in H^5(0,L)$.  Integrating by parts in the left hand side
of \eqref{a2}, we obtain that
\begin{equation*}
\begin{split}
&\left| a\theta(0)u_{xx}(0)-b\theta_x(L)u_{xxx}(L)+b\theta(L)u_{xxxx}(L)-b\theta_{xx}(0)u_{xx}(0)+b\theta_x(0)u_{xxx}(0) \right.\\
&\left.-b\theta(0)u_{xxxx}(0)+ av(L)\eta_{xx}(L)-bv_{xx}(L)\eta_{xx}(L)+bv_x(L)\eta_{xxx}(L) -bv(L)\eta_{xxxx}(L) \right. \\
&\left. -bv_x(0)\eta_{xxx}(0)-bv(0)\eta_{xxxx}(0)\right| \leq C \left( \int_0^L (\eta^2+u^2)dx\right)^{\frac12},
\end{split}
\end{equation*}
for all $(\eta,u) \in D(A)$. Then, it follows that
\begin{align*}
\begin{cases}
\theta(0)=\theta(L)=\theta_x(0)=\theta_x(L)=\theta_{xx}(0)=0, \\
v(0)=v(L)=v_x(0)=v_x(L)=v_{xx}(L)=0.
\end{cases}
\end{align*}
Hence $(\theta, v) \in D(\widetilde{A}) = D(-\widetilde{A})$. Thus, $D(\widetilde{A}^{*} ) = D(-\widetilde{A})$ and $\widetilde{A}^{*} = -\widetilde{A}.$
\end{proof}

\begin{corollary}\label{coro2.2}
For any $(\eta_0,u_0) \in \overline{X}_0$, system \eqref{new1} admits a unique solution $(\eta,u) \in C(\R; \overline{X}_0)$, which satisfies $\|(\eta(t),u(t))\|_{\overline{X}_0} = \|(\eta_0,u_0)\|_{\overline{X}_0}$ for all $t \in \R$. If, in addition, $(\eta_0, u_0) \in \overline{X}_5$, then $(\eta, u) \in C(\R; \overline{X}_5)$ with $\|(\eta, u)\|_{\overline{X}_5} := \|(\eta, u)\|_{\overline{X}_0} + \|\widetilde{A}(\eta, u)\|_{\overline{X}_0}$ constant.
\end{corollary}

Using the above Corollary  combined with some interpolation argument between $\overline{X}_0$ and $\overline{X}_5$, we can deduce that, for any $s \in (0, 5)$, there exists a constant $C_s > 0$ such that, for any $(\eta_0, u_0) \in \overline{X}_s$, the solution $(\eta, u)$ of \eqref{new1} satisfies $(\eta, u) \in C(\R;\overline{X}_s)$ and
\begin{equation}\label{inter1}
\|(\eta(t), u(t))\|_{\overline{X}_s} \leq  C_s \|(\eta_0, u_0)\|_{\overline{X}_s}, \quad \forall t \in \R.
\end{equation}

Now, we put our attention in the existence of traces. Indeed, we know that the traces
\begin{align*}
\eta(0,t), \,\, \eta(L,t), \,\, \eta_{x}(0,t), \,\, \eta_{x}(L,t), \,\, \eta_{xx}(0,t) , \\
u(0,t), \,\, u(L,t), \,\, u_{x}(0,t), \,\, u_{x}(L,t), \,\, u_{xx}(L,t),
\end{align*}
vanish. Thus, we have a look at the other traces $\eta_{xx}(0,t)$ and $u_{xx}(L,t)$.
\begin{proposition}
Let $(\eta_0,u_0) \in \overline{X}_{2}$ and let $(\eta, u)$ denote the solution of \eqref{new1}. Pick any $T>0$. Then $\eta_{xx}(L,t), u_{xx}(0,t) \in L^2(0,T)$ with
\begin{equation}
\label{new2}
\int_0^T \left( |\eta_{xx}(L,t)|^2 +|u_{xx}(0,t)|^2\right) dt \leq C \|(\eta_0,u_0)\|_{\overline{X}_2}^2
\end{equation}
for some constant $C=C(L,T,a,b)$.
\end{proposition}

\begin{proof}
Assume that $(\eta_0, u_0) \in X_3$, so that $(\eta , u) \in C([0,T];X_3)\cap L^1([0,T]; X_0)$. We multiply the first (resp. second) equation in \eqref{new1} by
$xu$ (resp. $x\eta$), integrate over $(0,T)\times (0,L)$, integrate by parts  and add the two obtained equations to get
\begin{multline}
-\frac{5b}{2} \int_0^T \int_0^L [\eta_{xx}^2 + u_{xx}^2] dxdt
-\frac{3a}{2} \int_0^T \int_0^L [\eta _x^2 + u_x^2] dxdt
-\frac{1}{2} \int_0^T \int_0^L [\eta ^2 + u^2] dxdt \\
+ \left[\int_0^L [x\eta u] dx \right] _0^T
+\frac{bL}{2}\int_0^T(|\eta_{xx}(L,t)|^2+|u_{xx}(L,t)|^2)dt =0.
\end{multline}
Since $\int_0^T\|(\eta ,u)\| ^2_{\overline{X}_2} dt \le C\| (\eta_0, u_0)\|^2_{\overline{X}_2}$, this yields
\[
\int_0^T |\eta_{xx}(L,t)|^2 (t,L)\, dt \le C \|(\eta _0, u_0)\|^2 _{\overline{X}_2}.
\]
By symmetry, using now as multipliers $(L-x)u$ and $(L-x)\eta$, we infer that \
\[
\int_0^T |u_{xx}(t,0)|^2\, dt \le C \Vert (\eta _0, u_0)\Vert ^2 _{\overline{X}_2}.
\]
Thus, \eqref{new2} is established when $(\eta _0, u_0)\in \overline{X}_3$. Since
$\overline{X}_3$  is dense in $\overline{X}_2$, the result holds as well for $(\eta _0, u_0)\in \overline{X}_2$.
\end{proof}

\begin{proposition}\label{propA2}
Let $(\eta_0,u_0) \in \overline{X}_{3}$ and let $(\eta, u)$ denote the solution of \eqref{new1}. Then $\eta_{xx}(L,t), u_{xx}(0,t) \in H^{\frac15}(0,T)$ with
\begin{equation}
\label{new5}
\|\eta_{xx}(L,\cdot)\|^2_{H^{\frac15}(0,T)} +\|u_{xx}(0,\cdot)\|^2_{H^{\frac15}(0,T)} \leq C \|(\eta_0,u_0)\|_{\overline{X}_3}^2
\end{equation}
for some constant $C=C(L,T,a,b)$.
\end{proposition}
\begin{proof}
In direction to prove \eqref{new5}, we consider $(\eta_0,u_0) \in \overline{X}_7$.   By Proposition \ref{propA}, $\widetilde{A}$ generates a group of isometries. Thus, by semigroup properties (see \cite{pazy1983}) we obtain that $(\eta, u) \in C(\R;\overline{X}_7)$, so that
\begin{equation}
(\widehat{\eta},\widehat{u})=(\eta_t,u_t)=\widetilde{A}(\eta,u) \in C(\R;\overline{X}_2)
\end{equation}
and it solves
\begin{equation}
\begin{cases}
(\widehat{\eta},\widehat{u})_t=\widetilde{A}(\widehat{\eta},\widehat{u}), \\
(\widehat{\eta},\widehat{u})(0)=\widetilde{A}(\eta_0,u_0) \in \overline{X}_2.
\end{cases}
\end{equation}
The, from \eqref{new2}, we deduce that
\begin{equation}
\label{new9}
\|\eta_{xx}(\cdot,L)\|^2_{H^1(0,T)}+\|u_{xx}(\cdot,0)\|^2_{H^1(0,T)}\leq \|(\eta_0,u_0)\|^2_{\overline{X}_7}.
\end{equation}
Since $\overline{X}_{3}=[\overline{X}_2,\overline{X}_7]_{\frac15}$, we infer from \eqref{new2} and \eqref{new9} that
\begin{equation*}
\|\eta_{xx}(\cdot,L)\|^2_{H^{\frac15}(0,T)}+\|u_{xx}(\cdot,0)\|^2_{H^{\frac15}(0,T)}\leq \|(\eta_0,u_0)\|^2_{\overline{X}_{3}},
\end{equation*}
for some constant $C=C(T)$ and all $(\eta_0,u_0) \in \overline{X}_3$.
\end{proof}

\section{Well-posedness: Nonlinear system}\label{Sec2}
In this section we prove the well-posedness for the nonlinear system
\begin{equation}
\begin{cases}\label{n1}
\eta_t + u_x-au_{xxx}+ a_1 (\eta u)_x +a_2(\eta u_{xx})_x + b u_{xxxxx}= 0,\\
u_t +\eta_x -a\eta_{xxx} +a_1uu_x+a_3(\eta\eta_{xx})_{x}+  a_4u_xu_{xx}    +b\eta_{xxxxx}=0, \\
\eta(x,0)= \eta_0(x), \quad u(x,0)=  u_0(x),
\end{cases}
\end{equation}
with $a>0$, $b>0$, $a\neq b$, $a_1>0$, $a_2<0$, $a_3>0$ and $a_4>0$, with the following boundary conditions
\begin{equation}
\begin{cases}\label{n2}
\eta(0,t)=\eta(L,t)=\eta_{x}(0,t)=\eta_x(L,t)=0,\\
u(0,t)=u(L,t)=u_{x}(0,t)=u_x(L,t)=0, \\
u_{xx}(0,t)+\alpha_1\eta_{xx}(0,t)=0, \,\,\,\, u_{xx}(L,t)-\alpha_2\eta_{xx}(L,t)=0, \quad \alpha_1,\alpha_2>0
\end{cases}
\end{equation}

Before presenting the proof of the main theorem  of this section, it is necessary to establish some definition to show how the solution of the problem \eqref{n1}-\eqref{n2} is obtained.

\begin{definition}\label{transposition}
Given $T>0$, $(\eta_0,u_0) \in \overline{X}_3$, $(h_1,h_2) \in L^2(0,T;X_{-2})$ and $f, g \in H^{-\frac15}(0,T)$, consider the non-homogeneous system
\begin{equation}
\label{nhs1}
\begin{cases}
\eta_t + u_x-au_{xxx}+bu_{xxxxx}= h_1, & \text{in} \,\, (0,L)\times (0,T),\\
u_t +\eta_x -a\eta_{xxx} +b\eta_{xxxxx}=h_2,  & \text{in} \,\, (0,L)\times (0,T), \\
\eta(0,t)=\eta(L,t)=\eta_x(0,t)=\eta_x(L,t)=0,\,\,\eta_{xx}(0,t)=f(t),& \text{on} \,\, (0,T), \\
u(0,t)=u(L,t)=u_x(0,t)=u_x(L,t)=0,\,\,u_{xx}(L,t)=g(t), & \text{on} \,\, (0,T), \\
\eta(x,0)= \eta_0(x), \quad u(x,0)=  u_0(x), & \text{on} \,\, (0,L).
\end{cases}
\end{equation}
A solution of the problem \eqref{nhs1} is a function $(\eta,u)$ in $C([0,T];\overline{X}_3)$ such that, for any $\tau \in [0,T]$ and $(\varphi_\tau,\psi_\tau) \in \overline{X}_3$, the following identity holds
\begin{equation}\label{transp1}
\begin{split}
\left((\eta(\tau),u(\tau)),(\varphi_\tau,\psi_\tau)\right)_{\overline{X}_3} =& \left((\eta_0,u_0),(\varphi(0),\psi(0))\right)_{\overline{X}_3} \\
&+\left\langle f(t),\chi_{(0,\tau)}(t)\psi_{xx}(0,t)\right\rangle_{H^{-\frac15}(0,T),H^{\frac15}(0,T)}  \\
&+ \left\langle g(t),\chi_{(0,\tau)}(t)\varphi_{xx} (L,t) \right\rangle_{H^{-\frac15}(0,T),H^{\frac15}(0,T)}  \\&+ \int_0^\tau \left\langle ( h_1(t), h_2(t)), (\varphi(t),\psi(t))\right\rangle_{(\overline{X}_{-2}, \overline{X}_2)^2} dt,
\end{split}
\end{equation}
where $\left(\cdot,\cdot\right)_{\overline{X}_3}$ is the inner product of $\overline{X}_3$, $\left\langle\cdot,\cdot\right\rangle$ is the duality of two spaces, $\chi_{(0,\tau)}(\cdot)$ denotes the characteristic function of the interval $(0,\tau)$ and $(\varphi,\psi)$ is the solution of
\begin{equation}\label{adjoint'}
\begin{cases}
\varphi_t +\psi_x-a\psi_{xxx}+b\psi_{xxxxx}= 0, & \text{in} \,\, (0,L)\times (0,\tau),\\
\psi_t +\varphi_x -a\varphi_{xxx}+b\varphi_{xxxxx}=0  & \text{in} \,\, (0,L)\times (0,\tau), \\
\varphi(0,t)=\varphi(L,t)=\varphi_{x}(0,t)=\varphi_{x}(L,t)=\varphi_{xx}(0,t)=0,& \text{on} \,\,  (0,\tau),\\
\psi(0,t)=\psi(L,t)=\psi_{x}(0,t)=\psi_{x}(L,t)=\psi_{xx}(L,t)=0, & \text{on} \,\,  (0,\tau),\\
\varphi(x,\tau)=\varphi_\tau,\quad \psi(x,\tau)=\psi_\tau, & \text{on} \,\, (0,L).
\end{cases}
\end{equation}
\end{definition}
The well-posedness of \eqref{adjoint'} is guaranteed by Corollary \ref{coro2.2} and \eqref{inter1}.
\begin{remark}
Note that the  right hand side of \eqref{transp1} is well defined for all $\tau \in [0, T]$, since  $\psi_{xx}(0,\cdot)$ and $\varphi_{xx}(L,\cdot)$ belong to $H^{\frac15}(0,\tau)$, by Proposition \ref{propA2}. The fact that $\chi_{(0,\tau)}\psi_{xx}(0,\cdot)$ and $\chi_{(0,\tau)}\varphi_{xx}(L,\cdot)$ belong to  $H^{\frac15}(0,T)$, for any $\tau \in [0, T]$, follows from \cite[Theorem 11.4, p. 60]{lions}.
\end{remark}

The next result borrowed from \cite{CaGa}, with minor changes, gives us the existence and uniqueness of solution for system \eqref{nhs1}.  Its proof is presented here for the sake of completeness.

\begin{lemma}\label{existentransposition1}
Let $T>0$, $(\eta_0,u_0) \in \overline{X}_3$, $(h_1,h_2) \in L^2(0,T;X_{-2})$ and $f,g \in H^{-\frac15}(0,T)$. There exists a unique solution $(\eta,u) \in C([0,T];\overline{X}_{3})$ of the system \eqref{nhs1}. Moreover, there exists a positive constant $C_T$, such that
\begin{equation}\label{depend1}
\begin{split}
\|(\eta(\tau),u(\tau))\|_{\overline{X}_3}\leq C_T \left(\|(\eta_0,u_0)\|_{\overline{X}_3}+\|f\|_{H^{-\frac15}(0,T)}+\|g\|_{H^{-\frac15}(0,T)} \right. \\       \left. \qquad + \|(h_1,h_2)\|_{L^2(0,T;X_{-2})}\right),
\end{split}
\end{equation}
for all $\tau \in [0,T]$.
\end{lemma}

\begin{proof}
Let $T>0$ and $\tau\in[0,T]$.  From Proposition \ref{propA}, $\widetilde{A}$ defined by \eqref{newA}-\eqref{newA1} is skew adjoint and generated a $C_0-$semigroup $\widetilde{S}(t)$. Note that making the change of variable $(x,t)\mapsto (\varphi(x,\tau-t),\psi(x,\tau-t))$ and taking $(\varphi_{\tau},\psi_{\tau}) \in \overline{X}_3$, we have that the solution of \eqref{adjoint'}  is given by $$(\varphi,\psi)=\widetilde{S}^*(\tau-t)(\varphi_\tau,\psi_\tau)=-\widetilde{S}(\tau-t)(\varphi_\tau,\psi_{\tau}).$$ Moreover, \eqref{inter1} implies that  $$(\varphi,\psi)\in  C(\R;\overline{X}_3).$$ In particular, there exists $C_T>0$, such that
\begin{equation}\label{adj1'}
\|(\varphi(t),\psi(t))\|_{\overline{X}_3}=\|\widetilde{S}^*(\tau-t)(\varphi_\tau,\psi_\tau)\|_{\overline{X}_3}\leq C_T\|(\varphi_\tau,\psi_{\tau})\|_{\overline{X}_3}, \quad \forall t\in [0,\tau].
\end{equation}
Let us define $L$ a linear functional given by the right hand side of \eqref{transp1}, that is,
\begin{equation*}
\begin{split}
L(\varphi_\tau,\psi_\tau)=& \left( (\eta_0,u_0),\widetilde{S}^*(\tau)(\varphi_\tau,\psi_\tau) \right)_{\overline{X}_3} \\
&+\left\langle (g(t),f(t)),\chi_{(0,\tau)}(t)\frac{d^2}{dx^2}(\widetilde{S}^*(\tau-t)(\varphi_\tau,\psi_\tau))\Big|^L_0\right\rangle_{(H^{-\frac15}(0,T), H^{\frac15}(0,T))^2} \\
&+ \int_0^\tau \left\langle(h_1(t), h_2(t)), \widetilde{S}^*(\tau-t)(\varphi_\tau,\psi_\tau)\right\rangle_{(\overline{X}_{-2},\overline{X}_2)^2} dt.
\end{split}
\end{equation*}
\begin{Claim*}
$L$ belongs to  $\LL(\overline{X}_3;\mathbb{R})$.
\end{Claim*}

Indeed, from the fact $\overline{X}_3 \subset \overline{X}_2$ and Proposition \ref{propA2}, we obtain that
\begin{align*}
|L(\varphi_\tau,\psi_\tau)|\leq  & C_T\|(\eta_0,u_0)\|_{\overline{X}_3}\|(\varphi_\tau,\psi_\tau))\|_{\overline{X}_3} + C_T\|(\varphi_\tau,\psi_\tau)\|_{\overline{X}_3}\|(h_1,h_2)\|_{L^1(0,T;\overline{X}_{-2})} \\
&  + C_T\|(f,g)\|_{(H^{-\frac15}(0,T))^2}\|(\varphi_{xx}(L),\psi_{xx}(0)\|_{(H^{\frac15}(0,\tau))^2} \\
\leq  & C_T\left(\|(\eta_0,u_0)\|_{\overline{X}_5} + \|(f,g)\|_{(H^{-\frac15}(0,T))^2} + \|(h_1,h_2)\|_{L^2(0,T;X_{-2})}\right)\|(\varphi_\tau,\psi_\tau))\|_{\overline{X}_3},
\end{align*}
where in the last inequality we use \eqref{adj1'}. Then, from Riesz representation Theorem, there exist one and only one $(\eta_\tau, u_\tau) \in  \overline{X}_3$ such that
\begin{equation}\label{eqn1}
\left( (\eta_\tau,u_\tau),(\varphi_\tau,\psi_\tau)\right)_{\overline{X}_3}=L(\varphi_\tau,\psi_\tau), \quad \text{with} \quad  \|(\eta_\tau,u_\tau)\|_{\overline{X}_3}=\|L\|_ {\LL(\overline{X}_3;\mathbb{R})}
\end{equation}
and the uniqueness of the solution to the problem \eqref{nhs1} holds.

We prove now that the solution of the system \eqref{nhs1} satisfies  \eqref{depend1}. Let $(\eta,u) : [0, T] \to \overline{X}_3$ be defined by
\begin{equation}\label{eqn2}
(\eta(\tau),u(\tau)):=(\eta_\tau,u_\tau), \quad \forall \tau \in [0,T].
\end{equation}
From \eqref{eqn1} and \eqref{eqn2}, \eqref{transp1} follows and
\begin{multline*}
\|(\eta(\tau),u(\tau))\|_{\overline{X}_3}= \|L\|_{ \LL(\overline{X}_3;\mathbb{R})} \leq C_T\left(\|(\eta_0,u_0)\|_{\overline{X}_3} + \|f\|_{(H^{-\frac15}(0,T)]}\right. \\
\left.+ \|g\|_{(H^{-\frac15}(0,T)]} + \|(h_1,h_2)\|_{L^2(0,T;\overline{X}_{-2})}\right).
\end{multline*}
In order to prove that the solution  $(\eta, u)$ belongs to $ C([0,T];\overline{X}_3)$, let $\tau \in [0,T]$ and $\{\tau_n\}_{n\in \N}$ be a sequence such that
\begin{equation*}
\tau_n \longrightarrow \tau, \quad \text{as $n\rightarrow \infty$.}
\end{equation*}
Consider $(\varphi_\tau,\psi_\tau) \in \overline{X}_3$ and $\{(\varphi_{\tau_n},\psi_{\tau_n})\}_{n\in \N}$ be a sequence in $\overline{X}_3$ such that
\begin{equation}\label{eqn4}
(\varphi_{\tau_n},\psi_{\tau_n}) \rightarrow  (\varphi_{\tau},\psi_\tau) \quad \text{strongly in $\overline{X}_3$, as $n \rightarrow \infty$.}
\end{equation}
Note that
\begin{align}
\lim_{n\rightarrow \infty}\left( (\eta_0,w_0), \widetilde{S}^*(\tau_n)(\varphi_{\tau_n},\psi_{\tau_n})\right)_{\overline{X}_3}=  \left((\eta_0,w_0),\widetilde{S}^*(\tau) (\varphi_{\tau},\psi_{\tau})\right)_{\overline{X}_3}. \label{eqn5}
\end{align}
Indeed,
\begin{equation*}
\begin{split}
\lim_{n\rightarrow \infty}\left( (\eta_0,w_0), \widetilde{S}^*(\tau_n)(\varphi_{\tau_n},\psi_{\tau_n})\right)_{\overline{X}_3}
=&  \lim_{n\rightarrow \infty}\left( (\eta_0,w_0), \widetilde{S}^*(\tau_n)\left( (\varphi_{\tau_n},\psi_{\tau_n})-(\varphi_{\tau},\psi_{\tau})\right) \right)_{\overline{X}_3} \\
& + \lim_{n\rightarrow \infty}\left( (\eta_0,w_0), \widetilde{S}^*(\tau_n)(\varphi_{\tau},\psi_{\tau})\right)_{\overline{X}_3}.
 \end{split}
\end{equation*}
From \eqref{eqn4} and since $\{\widetilde{S}(t)\}_{t \geq0}$ is a strongly continuous group of continuous linear operators on $X_0$, we have
\begin{align*}
\lim_{n\rightarrow \infty}\left( (\eta_0,w_0), \widetilde{S}^*(\tau_n)\left( (\varphi_{\tau_n},\psi_{\tau_n})-(\varphi_{\tau},\psi_{\tau})\right) \right)_{\overline{X}_3}=0
\end{align*}
and consequently,
\begin{align*}
\lim_{n\rightarrow \infty}\left( (\eta_0,w_0), \widetilde{S}^*(\tau_n)(\varphi_{\tau},\psi_{\tau})\right)_{\overline{X}_3}&=\left( (\eta_0,w_0), \widetilde{S}^*(\tau)(\varphi_{\tau},\psi_{\tau})\right)_{\overline{X}_3}.
\end{align*}
Thus, \eqref{eqn5} follows. Now, we have to analyze the following limits,
\begin{align}
&\lim_{n\rightarrow \infty}\left\langle (g(t),f(t)),\chi_{(0,\tau_n)}(t)\frac{d^2}{dx^2}(\widetilde{S}^*(\tau_n-t)(\varphi_{\tau_n},\psi_{\tau_n}))\Big|_0^L\right\rangle_{(H^{-\frac15}(0,T),H^{\frac15}(0,T))^2} \label{newimp1}
\end{align}
and
\begin{align}
&\lim_{n\rightarrow \infty} \int_0^\tau\left\langle (h_1(t),h_2(t)), \widetilde{S}^*(\tau_n-t)(\varphi_{\tau_n},\psi_{\tau_n}) \right\rangle_{(\overline{X}_{-2}, \overline{X}_2)^2} dt. \label{newimp2}
\end{align}
In fact, observe that, by group properties of $\widetilde{S}^*$ and Proposition \ref{propA2}, we have that
\begin{align*}
    \left\|\frac{d^2}{dx}\widetilde{S}^*(\tau-t)(\varphi_{\tau},\psi_{\tau})\Big|_0^L\right\|_{[H^{\frac15}(0,\tau)]^2} \leq C \|(\varphi_{\tau},\psi_{\tau})\|_{\overline{X}_3}.
\end{align*}
Thus, the linear map $(\varphi_\tau,\psi_\tau) \in \overline{X}_3 \mapsto \frac{d^2}{dx^2}(\widetilde{S}^*(\tau-\cdot)(\psi_{\tau},\varphi_{\tau}))\Big|_0^L$ belongs to $H^{\frac15}(0,\tau; \R^2)$ and it is continuous. Moreover, as the natural extension by $0$ outside $(0,\tau)$ is a continuous mapping from $H^{\frac15}(0,\tau)$ into $H^{\frac15}(0,T)$ (cf. \cite[Theorem 11.4, p. 60]{lions}), we obtain that the map $(\varphi_\tau,\psi_\tau) \in \overline{X}_3 \mapsto \chi_{(0,\tau_n)}(\cdot)\frac{d^2}{dx^2}(\widetilde{S}^*(\tau-\cdot)(\psi_{\tau},\varphi_{\tau}))\Big|_0^L$ belongs to $H^{\frac15}(0,T; \R^2)$ and it is continuous, as well. Since a continuous linear map between two Hilbert spaces is weakly continuous, \eqref{eqn4} implies that
\begin{equation}\label{newimp3}
\chi_{(0,\tau_n)}(t)\frac{d^2}{dx^2}(\widetilde{S}^*(\tau_n-\cdot)(\varphi_{\tau_n},\psi_{\tau_n}))\Big|_0^L \rightharpoonup  \chi_{(0,\tau)}(t)\frac{d^2}{dx^2}(\widetilde{S}^*(\tau-\cdot)(\varphi_\tau,\psi_\tau))\Big|_0^L,
\end{equation}
weakly in $H^{\frac15}([0,T];\R^2)$, as $n\rightarrow\infty$. Thus, by using \eqref{newimp3}, the limit \eqref{newimp1} yields that
\begin{multline}\label{eqn10}
\lim_{n\rightarrow\infty}  \left\langle(g(t),f(t)),\chi_{(0,\tau_n)}(t)\frac{d^2}{dx^2}(\widetilde{S}^*(\tau_n-t)(\varphi_{\tau_n},\psi_{\tau_n}))\Big|_0^L\right\rangle_{(H^{-\frac15}(0,T),H^{\frac15}(0,T))^2} \\
=\left\langle(g(t),f(t)),\chi_{(0,\tau)}(t)\frac{d^2}{dx^2}(\widetilde{S}^*(\tau-t)(\varphi_{\tau},\psi_{\tau}))\Big|_0^L\right \rangle_{(H^{-\frac15}(0, T),H^{\frac15}(0,T))^2}.
\end{multline}
On the other hand, extending by zero the functions $h_i$, for $i=1,2$, we obtain elements of $[H^{\frac15}(-T,T)]'$ and $L^2(-T,T;\overline{X}_{-2})$, that is, $$h_i \equiv 0\quad\text{a.e \quad in  $(-T,0)\times(0,L)$}, $$
and setting $s=\tau_n-t$, we have that
\begin{multline}\label{newimp6}
\int_0^{\tau_n}\left\langle(h_1(t),h_2(t)),\widetilde{S}^*(\tau_n-t)(\varphi_{\tau_n},\psi_{\tau_n})\right\rangle_{(\overline{X}_{-2}, \overline{X}_{2})^2}dt \\
=\int_0^T \chi_{(0,\tau_n)}(s)\left\langle(h_1(\tau_n-s),h_2(\tau_n-s)),\widetilde{S}^*(s)(\varphi_{\tau_n},\psi_{\tau_n})\right\rangle_{(\overline{X}_{-2}, \overline{X}_{2})^2}dt.
\end{multline}
Similarly, taking $s=\tau-t$ in \eqref{newimp2}, we get
\begin{multline}\label{newimp7}
\int_0^\tau\left\langle (h_1(t),h_2(t)), \widetilde{S}^*(\tau-t)(\varphi_\tau,\psi_\tau) \right\rangle_{(\overline{X}_{-2}, \overline{X}_2)^2} dt \\
=\int_0^T \chi_{(0,\tau)}(s)\left\langle( h_1(\tau-s),h_2(\tau-s)), \widetilde{S}^*(s)(\varphi_\tau,\psi_\tau) \right\rangle_{(\overline{X}_{-2}, \overline{X}_2)^2}dt.
\end{multline}
Since the translation in time is continuous in $L^2(0,T; X_{-2})$ and using the dominated convergence theorem, we obtain
\begin{equation}\label{newimp4}
\chi_n(\cdot)(h_1(\tau_n-\cdot,\cdot),h_2(\tau_n-\cdot,\cdot)) \longrightarrow \chi(\cdot)(h_1(\tau-\cdot,\cdot),h_2(\tau-\cdot,\cdot)),
\end{equation}
in $L^2(0,T; X_{-2})$, as $n\rightarrow\infty$.
Similarly, by the strong continuity of the group, it follows that
\begin{align*}
\widetilde{S}^*(\cdot)(\varphi_{\tau_n},\psi_{\tau_n}) \rightharpoonup   \widetilde{S}^*(\cdot)(\varphi_\tau,\psi_\tau)
\end{align*}
weakly in $L^2(-T,T;X_0)$, as $n\rightarrow\infty$. In particular, we obtain that
\begin{equation}\label{newimp5}
\widetilde{S}^*(\cdot)(\varphi_{\tau_n},\psi_{\tau_n}) \rightharpoonup   \widetilde{S}^*(\cdot)(\varphi_\tau,\psi_\tau),
\end{equation}
weakly in $L^2(-T,T;\overline{X}_2)$, as $n\rightarrow\infty$. By using \eqref{newimp6}, \eqref{newimp7}, \eqref{newimp4} and \eqref{newimp5}, the limit \eqref{newimp2} yields that
\begin{multline}\label{eqn10'}
\lim_{n\rightarrow\infty}\int_0^{\tau_n} \left\langle (h_1(t),h_2(t)),\widetilde{S}^*(\tau_n-t)(\varphi_{\tau_n},\psi_{\tau_n})\right\rangle_{(\overline{X}_{-2},
\overline{X}_2)^2} dt \\
=\int_0^\tau \left\langle(h_1(t),h_2(t)), \widetilde{S}^*(\tau-t)(\varphi_{\tau},\psi_{\tau})\right\rangle_{(\overline{X}_{-2},
\overline{X}_2)^2}dt .
\end{multline}
Finally, from \eqref{eqn1}, \eqref{eqn2}, \eqref{eqn5}, \eqref{eqn10} and \eqref{eqn10'}, one gets
\begin{align*}
\left( (\eta(\tau_n),w(\tau_n)),(\varphi_{\tau_n},\psi_{\tau_n})\right)_{\overline{X}_3 }\longrightarrow \left( (\eta(\tau),w(\tau)),(\varphi_{\tau},\psi_{\tau}))\right)_{\overline{X}_3}, \quad \text{as $n\rightarrow\infty,$}
\end{align*}
which implies that
\begin{align*}
\left(\eta(\tau_n),w(\tau_n)\right) \longrightarrow \left( \eta(\tau),w(\tau)\right) \quad \text{in \quad $\overline{X}_3$, \quad as $n\rightarrow\infty$}.
\end{align*}
This concludes the proof.
\end{proof}

The next result establishes the well-posedness of the non-homogeneous feedback linear system associated to \eqref{nhs1}.
\begin{lemma}\label{exitencefeedback1}
Let $T>0$. Then, for every $(\eta_0,u_0)$ in $\overline{X}_3$ and $(h_1,h_2)$ in $L^2(0,T;X_{-2})$, there exists a unique solution $(\eta, u)$ of the system \eqref{nhs1} such that  $$(\eta, u)\in C([0,T];\overline{X}_3),$$  with $f(t):=-\alpha_1\eta_{xx}(0,t)$ and $g(t):=\alpha_2\eta_{xx}(L,t)$, where $\alpha_1$ and $\alpha_2$ belong to $\R$.  Moreover, for some positive constant $C=C(T)$, we have
\begin{align*}
\|(\eta(t),u(t))\|_{\overline{X}_3} \leq C\left( \|(\eta_0,u_0)\|_{\overline{X}_3} + \|(h_1,h_2)\|_{L^2(0,T;X_{-2})}\right), \quad \forall t \in [0,T].
\end{align*}
\end{lemma}

\begin{proof}
Firstly, note that if $(\eta,u) \in C([0,T];\overline{X}_3)$, then $$f(t):=-\alpha_1\eta_{xx}(0,t) \text{ and } g(t):=\alpha_2\eta_{xx}(L,t)\in H^{-\frac15}(0,T).$$
In fact, by using the continuous embedding $L^{2}(0,T)\subset H^{-\frac15}(0,T)$ and the trace theorem \cite[Theorem 7.53]{adams}, there exists a positive constant $C:=C(L,\alpha_1,\alpha_2)$ such that
\begin{equation}\label{eee2}
\begin{split}
\|\alpha_1\eta_{xx}(0,\cdot)\|_{{H^{-\frac15}(0,T)}}^2+\|\alpha_2\eta_{xx}(L,\cdot)\|_{{H^{-\frac15}(0,T)}}^2
&\leq C \int_0^T \left( \alpha_1^2\eta_{xx}^2(0,t)+\alpha_2^2\eta_{xx}^2(L,t)\right)dt \\
& \leq C\int_0^T\|\eta(t)\|_{\overline{X}_3}^2dt \\
&\leq T C \|(\eta,u)\|_{C([0,T];\overline{X}_3)}^2.
\end{split}
\end{equation}

Let $0<\beta \leq T$ that will be determinate later. For each $(\eta_0,u_0) \in \overline{X}_3$, consider the map
\begin{align*}
\begin{array}{l c l c}
\Gamma : & C([0,\beta];\overline{X}_3) & \longrightarrow & C([0,\beta];\overline{X}_3) \\
         & (\eta,u) & \longmapsto & \Gamma(\eta,u)=(w,v)
\end{array}
\end{align*}
where, $(w,v)$ is the solution of the system \eqref{nhs1} with $f(t)=-\alpha_1\eta_{xx}(0,t)$ and $ g(t)=\alpha_2\eta_{xx}(L,t)$. By Lemma \ref{existentransposition1} and \eqref{eee2}, the linear map $\Gamma$ is well defined. Furthermore, there exists a positive constant $C_{\beta}$, such that
\begin{multline*}
\|\Gamma(\eta,u)\|_{C([0,\beta];\overline{X}_3)} \leq C_{\beta} \left(\|(\eta_0,u_0)\|_{\overline{X}_3} +\|(\alpha_1\eta_{xx}(0,t),\alpha_2\eta_{xx}(L,t))\|_{(H^{-\frac15}(0,\beta))^2} \right.\\
\left. + \|(h_1,h_2)\|_{L^2(0,T;\overline{X}_{-2})} \right).
\end{multline*}
Then,
$$
\|\Gamma(\eta,u)\|_{C([0,\beta];\overline{X}_3)} \leq C_T \left( \|(\eta_0,u_0)\|_{\overline{X}_3}+ \|(h_1,h_2)\|_{L^2(0,T;\overline{X}_{-2})} \right) + C_T\beta^{\frac12} \|(\eta,w)\|_{C([0,\beta];\overline{X}_3)}.
$$
Let $$(\eta,u) \in B_{R}(0):=\{ (\eta,u)\in C([0,\beta];\overline{X}_3): \|(\eta,u)\|_{C([0,\beta];\overline{X}_3)} \leq R\},$$ with $$R=2C_T\left( \|(\eta_0,u_0)\|_{\overline{X}_3}+ \|(h_1,h_2)\|_{L^2(0,T;\overline{X}_{-2})} \right).$$ Choosing $\beta$ such that $
C_T\beta^{\frac12}\leq \frac12,$ 
it implies that $\|\Gamma(\eta,u)\|_{C([0,\beta];\overline{X}_3)} \leq R$, for all $(\eta,u) \in B_{R}(0)$, i.e,  $\Gamma$ maps $B_R(0)$ into $B_R(0)$. 

On the other hand, note that
\begin{align*}
\|\Gamma(\eta_1,u_1)-\Gamma(\eta_2,u_2)\|_{C([0,\beta];\overline{X}_3)}
\leq & \ C_T\beta^{\frac12}\|(\eta_1-\eta_2,u_1-u_2)\|_{C([0,\beta];\overline{X}_3)} \\
\leq &  \  \frac12 \|(\eta_1-\eta_2,u_1-u_2)\|_{C([0,\beta];\overline{X}_3)}.
\end{align*}
Hence, $\Gamma: B_R(0) \longrightarrow B_R(0)$ is a contraction and, by Banach fixed point theorem, we obtain a unique $(\eta, u) \in  B_R(0)$, such that $\Gamma(\eta,u)=(\eta,u)$ and
\begin{align*}
\|(\eta,u)\|_{C([0,\beta];\overline{X}_3)} \leq 2C_T\left( \|(\eta_0,u_0)\|_{\overline{X}_3}+ \|(h_1,h_2)\|_{L^2(0,T;\overline{X}_{-2})} \right).
\end{align*}
Since the choice of  $\beta$ is independent of $(\eta_0, u_0)$, the standard continuation extension argument yields that the solution $(\eta,u)$ belongs to $C([0,\beta];\overline{X}_3)$, thus, the proof is complete.
\end{proof}
We are now in position to prove one of the main result of this article.
\subsection{Proof of Theorem \ref{main_int1}}
Let $T>0$ and $\|(\eta_0,u_0)\|_{\overline{X}_3}<\rho$, where $\rho>0$ will be determined later. Note that for $(\eta,u) \in C([0,T];\overline{X}_3)$, there exists a positive constant $C_1$ such that
\begin{equation}\label{eqn11}
\begin{split}
\|\eta u_x\|^2_{L^2(0,T; L^2(0,L))} &\leq \int_0^T \|\eta(t)\|^2_{L^\infty(0,L)} \|u_x(t)\|^2_{L^2(0,L)}dt\\& \leq C_1' \int_0^T \|\eta(t)\|^2_{H^{1}(0,L)} \|u(t)\|^2_{H^{1}(0,L)}dt\\
& \leq C_1^2 T \|(\eta,u)\|^4_{C([0,T]; \overline{X}_3)},
\end{split}
\end{equation}
\begin{equation}\label{eqn11a}
\begin{split}
\|\eta u_{xx}\|^2_{L^2(0,T; L^2(0,L))} &\leq \int_0^T \|\eta(t)\|^2_{L^\infty(0,L)} \|u_{xx}(t)\|^2_{L^2(0,L)}dt\\& \leq C_1' \int_0^T \|\eta(t)\|^2_{H^{2}(0,L)} \|u(t)\|^2_{H^{2}(0,L)}dt\\
& \leq C_1^2 T \|(\eta,u)\|^4_{C([0,T]; \overline{X}_3)},
\end{split}
\end{equation}
\begin{equation}\label{eqn11b}
\begin{split}
\|\eta_x u_{xx}\|^2_{L^2(0,T; L^2(0,L))} &\leq \int_0^T \|\eta_x(t)\|^2_{L^\infty(0,L)} \|u_{xx}(t)\|^2_{L^2(0,L)}dt\\& \leq C_1' \int_0^T \|\eta(t)\|^2_{H^{2}(0,L)} \|u(t)\|^2_{H^{2}(0,L)}dt\\
& \leq C_1^2 T \|(\eta,u)\|^4_{C([0,T]; \overline{X}_3)},
\end{split}
\end{equation}
and
\begin{equation}\label{eqn11c}
\begin{split}
\|\eta u_{xxx}\|^2_{L^2(0,T; L^2(0,L))} &\leq \int_0^T \|\eta(t)\|^2_{L^\infty(0,L)} \|u_{xxx}(t)\|^2_{L^2(0,L)}dt\\& \leq C_1' \int_0^T \|\eta(t)\|^2_{H^{3}(0,L)} \|u(t)\|^2_{H^{3}(0,L)}dt\\
& \leq C_1^2 T \|(\eta,u)\|^4_{C([0,T]; \overline{X}_3)}.
\end{split}
\end{equation}
This implies that, for any $(\eta,u) \in C([0,T]; \overline{X}_3)$ and $a_i \in \R$, with $i=1,2,3,4$, we have that
$$
-a_1(\eta u)_x-a_2(\eta u_{xx}), -a_1 uu_x-a_3(\eta\eta_{xx})_x-a_4u_xu_{xx} \in L^2(0,T;\overline{X}_0) \subset  L^2(0,T;\overline{X}_{-2}).
$$
Consider the following linear map
\begin{align*}
\begin{array}{c c c c}
\Gamma : & C([0,T]; \overline{X}_3) & \longrightarrow & C([0,T]; \overline{X}_3)\\
         & (\eta,u) & \longmapsto & \Gamma(\eta,w)=(\overline{\eta},\overline{u}),
\end{array}
\end{align*}
where $(\overline{\eta},\overline{u})$ is the solution of the system \eqref{nhs1} with $$(h_1,h_2)=(-a_1(\eta u)_x-a_2(\eta u_{xx}), -a_1 uu_x-a_3(\eta\eta_{xx})_x-a_4u_xu_{xx})$$ in $L^2(0,T;\overline{X}_{-2})$,  with $f(t):=-\alpha_1\eta_{xx}(0,t)$ and $g(t):=\alpha_2 \eta_{xx}(L,t)$.

\begin{Claim*}
The map $\Gamma$ is well-defined, maps $B_R(0)$ into itself and it is a contraction in a ball.
\end{Claim*}

Indeed, firstly note that Lemma \ref{exitencefeedback1} ensures that $\Gamma$ is well-defined, moreover, using Lemma \ref{existentransposition1}, there exists a positive constant $C_T$, such that
\begin{equation*}
\begin{split}
\|\Gamma(\eta,u)\|_{ C([0,T]; \overline{X}_3)} \leq &C_{T} \left(\|(\eta_0,u_0)\|_{\overline{X}_3}+\|a_1(\eta u)_x+a_2(\eta u_{xx})\|_{L^2(0,T; \overline{X}_{-2})}\right. \\
&\left. +\|a_1 uu_x+a_3(\eta\eta_{xx})_x+a_4u_xu_{xx}\|_{L^2(0,T; \overline{X}_{-2})}
+ \|w w_x\|_{L^2(0,T; \overline{X}_{-2})}\right).
\end{split}
\end{equation*}
Then, equations \eqref{eqn11}, \eqref{eqn11a}, \eqref{eqn11b} and \eqref{eqn11c}  yield that
\begin{equation}\label{Gamma}
\begin{split}
\|\Gamma(\eta,u)\|_{ C([0,T]; \overline{X}_3)} \leq &C_{T} \|(\eta_0,u_0)\|_{\overline{X}_3}\\
&+ (3|\alpha_1|+|a_2|+2|a_3|+|a_4|)T^{1/2}C_TC_1 \|(\eta,u)\|^2_{C([0,T];\overline{X}_3)} \\
\leq & C_{T} \|(\eta_0,u_0)\|_{\overline{X}_3} + 7MT^{1/2}C_TC_1 \|(\eta,u)\|^2_{C([0,T];\overline{X}_3)},
\end{split}
\end{equation}
where $M=\max \{ |a_1|,|a_2|,|a_3|,|a_4|\}$. Consider the ball  $$B_R(0) = \left\lbrace(\eta,u) \in C([0,T]; \overline{X}_3 ): \|(\eta,u)\|_{C([0,T]; \overline{X}_3)}\leq R\right\rbrace,$$ where $R= 2C_T\|(\eta_0,u_0)\|_{\overline{X}_3}.$ From the estimate \eqref{Gamma} we get that
$$
\|\Gamma(\eta,u)\|_{C([0,T];\overline{X}_3)} \leq \frac{R}{2}+ 7MT^{1/2}C_TC_1R^2
 <\frac{R}{2}+ 14MT^{1/2}C_T^2C_1\rho R,
$$
for all $(\eta,u) \in B_R(0)$. Consequently, if we choose $\rho>0$ such that
\begin{equation}\label{ee3}
14MT^{1/2}C_T^2C_1\rho <\frac14,
\end{equation}
$\Gamma$ maps the ball $B_R(0)$ into itself. Finally, note that
\begin{equation*}
\begin{split}
\|\Gamma(\eta_1,u_1)-\Gamma(\eta_2,u_2)\|_{{C([0,T];\overline{X}_3)}} \leq& C_T \|a_1((\eta_2u_2)_x-(\eta_1u_1)_x)\|_{L^2(0,T;\overline{X}_{-2} )}\\
&+ C_T \|a_2((\eta_2u_{2,xx})_x-(\eta_1u_{1,xx})_x)\|_{L^2(0,T;\overline{X}_{-2} )} \\
&+ C_T \|a_1(u_2u_{2,x}-u_1u_{1,x})\|_{L^2(0,T;\overline{X}_{-2} )} \\
&+ C_T \|a_3((\eta_2\eta_{2,xx})_x-(\eta_1\eta_{1,xx})_x)\|_{L^2(0,T;\overline{X}_{-2} )} \\
&+ C_T \|a_4(u_{2,x}u_{2,xx}-u_1u_{1,xx})\|_{L^2(0,T;\overline{X}_{-2} )}.
\end{split}
\end{equation*}
Thus, we obtain 
\begin{equation*}
\begin{split}
&\|\Gamma(\eta_1,u_1)-\Gamma(\eta_2,u_2)\|_{{C([0,T];\overline{X}_3)}} \\
&\leq 3T^{1/2}C_TC_1 M(\|\eta_1\|_{C([0,T];H^{3}(0,L))}+ \|\eta_2\|_{C([0,T];H^{3}(0,L))} )\|u_1-u_2\|_{C([0,T];H^{3}(0,L))} \\
&+ 3T^{1/2}C_TC_1 M(\|u_1\|_{C([0,T];H^{3}(0,L))}+ \|u_2\|_{C([0,T];H^{3}(0,L))} )\|\eta_1-\eta_2\|_{C([0,T];H^{3}(0,L))} \\
&+ 2T^{1/2}C_TC_1 M(\|u_1\|_{C([0,T];H^{3}(0,L))}+ \|u_2\|_{C([0,T];H^{3}(0,L))} )\|u_1-u_2\|_{C([0,T];H^{3}(0,L))}.
\end{split}
\end{equation*}
Finally, it follows that
\begin{equation*}
\begin{split}
\|\Gamma(\eta_1,u_1)-\Gamma(\eta_2,u_2)\|_{{C([0,T];\overline{X}_3)}} \leq& \ 14T^{1/2}C_TC_1 M R \|(\eta_1-\eta_2,u_1-u_2)\|_{C([0,T];\overline{X}_3)} \\
<\ & 28T^{1/2}C_T^2C_1 M \rho \|(\eta_1-\eta_2,u_1-u_2)\|_{C([0,T];\overline{X}_3)}
\end{split}
\end{equation*}
Therefore, from \eqref{ee3}, we get
$$
\|\Gamma(\eta_1,u_1)-\Gamma(\eta_2,u_2)\|_{C([0,T];\overline{X}_3)} \\
 \leq \frac{1}{2}\|(\eta_1-\eta_2,u_1-u_2)\|_{{C([0,T];\overline{X}_3)}},
$$
for all $(\eta,u) \in B_T(0)$. Hence, $\Gamma: B_R(0) \longrightarrow B_R(0)$ is a contraction and the claim is archived.

Thanks to Banach fixed point theorem, we obtain a unique $(\eta, u) \in  B_R$, such that $\Gamma(\eta,u)=(\eta,u)$ and
\begin{align*}
\|(\eta,w)\|_{C([0,T];\overline{X}_3)} \leq 2C_{T} \|(\eta_0,u_0)\|_{\overline{X}_3}.
\end{align*}
Thus, the proof is archived. \qed

\subsection{Well-posedness in time}
Adapting the proof of Theorem \ref{main_int1}, one can also prove, without any restriction over the initial data $(\eta_0,u_0)$, that there exist $T^* > 0$ and a solution $(\eta,u)$ of \eqref{n1}-\eqref{n2}, satisfying the initial condition $\eta(\cdot,0) = \eta_0(\cdot)$ and $u(\cdot,0)=u_0(\cdot)$. More precisely,

\begin{theorem}
Let $(\eta_0,u_0) \in \overline{X}_3$.  Then, there exists $T^*>0$ a  unique solution $(\eta, u) \in C([0,T^*]; \overline{X}_3)$ of \eqref{n1}-\eqref{n2}. Moreover
\begin{align*}
\|(\eta,u)\|_{C([0,T]; \overline{X}_3)} \leq C \|(\eta_0,u_0)\|_{ \overline{X}_3},
\end{align*}
for some positive constant $C=C(T^*)$.
\end{theorem}

Observe that if $(\eta_1, u_1)\in C([0,T_1];\overline{X}_3)$ and $(\eta_2, u_2)\in C([0,T_2];\overline{X}_3)$  are  the solutions given by the Theorem \ref{main_int1} with initial data $(\eta_0, u_0)$ and $(\eta_1(T_1), u_1(T_1))$, respectively, the function $(\eta,u):[0,T_1+T_2]\rightarrow \overline{X}_3$ defined by
\begin{align*}
(\eta(t),u(t))=\begin{cases}
(\eta_1(t), u_1(t))  & \text{if $t \in [0,T_1]$}, \\
(\eta_2(t-T_1), u_2(t-T_2))  & \text{if $t \in [T_1,T_1+T_2]$},
\end{cases}
\end{align*}
is the solution of the feedback system on interval $[0,T_1+T_2]$ with initial data $(\eta_0, u_0)$. This argument allows us to extend a local solution until a maximal interval, that is, for all $0<T<T_{\max} \leq \infty$ there exists a function $(\eta, u)\in C([0,T];\overline{X}_3)$, solution of the feedback system \eqref{n1}-\eqref{n2}. The following proposition, easily holds:
\begin{proposition}\label{extension}
Let $(\eta_0, u_0) \in \overline{X}_3$ and $(\eta, w)\in C([0,T];\overline{X}_3)$ solution of the feedback system, for all $0<T<T_{\max}$, with initial data $(\eta_0, u_0)$. Then, only one of the following assertions hold:
\begin{enumerate}
\item[(i)] $T_{\max}=\infty$;
\item[(ii)] If $T_{\max}<\infty$, then, $\lim_{t\rightarrow T_{\max}}\|(\eta(t),w(t))\|_{\overline{X}_3}=\infty$.
\end{enumerate}
\end{proposition}

\section{Exponential stability for the linearized system}\label{Sec3}
Let us now to prove Theorem \ref{main_int} concerning of exponential stability for the linear system \eqref{homo1}.

\begin{proof}[Proof of Theorem \ref{main_int}]
Theorem \ref{main_int} is a consequence of the following claim:

\vspace{0.2cm} \textit{There exists a constant $C>0$, such that
\begin{equation}\label{decay}
\|(\eta_0,u_0)\|_{X_0}^2 \leq C \int_0^T \left(|\eta_{xx}(L,t)|^2+|\eta_{xx}(0,t)|^2\right)dt,
\end{equation}
where $(\eta,u)$ is the solution of \eqref{homo1} given by Proposition \ref{existence}.}

\vspace{0.2cm}
Indeed,  if \eqref{decay} is true, we get
\begin{align*}
E(T)-E(0)\leq -\frac{E(0)}{C},
\end{align*}
where $E(t)$ is defined by \eqref{energy}. This implies that
\begin{align*}
E(T)\leq E(0) -\frac{E(0)}{C} \leq E(0) -\frac{E(T)}{C}.
\end{align*}
Thus,
\begin{align*}
E(T)\leq \left(\frac{C}{C+1}\right)E(0),
\end{align*}
which gives Theorem \ref{main_int} by using semigroup properties associated to the model.
\end{proof}

We will divide the proof of the observability inequality \eqref{decay} in three steps as follows:
\begin{proof}[Proof of \eqref{decay}]
\noindent\textbf{Step 1: {\em Compactness-uniqueness argument}}
\vglue 0.2cm
We argue by contradiction. Suppose that \eqref{decay} does not hold, then there exists a sequence $\{(\eta_{0,n},u_{0,n})\}_{n\in\N} \in X_0$, such that
\begin{equation}\label{e2}
1=\|(\eta_{0,n},u_{0,n})\|_{X_0}^2 > n \int_0^T \left(|\eta_{n,xx}(L,t)|^2+|\eta_{n,xx}(0,t)|^2\right)dt,
\end{equation}
where $(\eta_n(t),u_n(t))=S(t)(\eta_{0,n},u_{0,n})$. Thus, from \eqref{e2} we obtain
\begin{equation}\label{e3}
\lim_{n\rightarrow \infty}\int_0^T \left(|\eta_{n,xx}(L,t)|^2+|\eta_{n,xx}(0,t)|^2\right)dt=0.
\end{equation}
Estimate \eqref{kato1} in Proposition \ref{prop2}, together with \eqref{e2}, imply that the sequence $\{(\eta_n,u_n)\}_{n\in\N}$ is bounded in $ L^2(0,T;X_2)$. Furthermore, by \eqref{homo1} we deduce that $\{(\eta_{n,t},u_{n,t})\}_{n\in\N}$ is bounded in $L^2(0,T;X_{-3})$. Thus, the compact embedding
\begin{equation}
X_2 \hookrightarrow X_0 \hookrightarrow X_{-3},
\end{equation}
allows us to  conclude that $\{ (\eta_n,u_n)\}_{n \in \N}$ is relatively compact in $L^2(0,T;X_0)$ and, consequently, we obtain a subsequence, still denoted by the same index $n$, satisfying
 \begin{equation}\label{e4}
(\eta_n,u_n) \rightarrow (\eta,u) \mbox{ in } L^2(0,T;X_0), \mbox{ as } n\rightarrow\infty.
 \end{equation}
Moreover, using \eqref{energia2}, \eqref{e3} and \eqref{e4}, we obtain that $\{(\eta_{0,n},u_{0,n})\}_{n\in \N}$ is a Cauchy sequence in $X_0$. Hence, there exists $(\eta_0,u_0) \in X_0$, such that
 \begin{equation}\label{e5}
(\eta_{0,n},u_{0,n}) \rightarrow (\eta_0,u_0) \mbox{ in } X_0, \mbox{ as } n\rightarrow\infty,
 \end{equation}
and, from \eqref{e2} we get $\|(\eta_0,u_0)\|_{X_0}=1$. On the other hand, note that combining \eqref{energia1}, \eqref{e3} and \eqref{e5}, we obtain a 
subsequence $\{(\eta_n,u_n)\}_{n\in\N}$, such that
 \begin{equation}
(\eta_n,u_n) \rightarrow (\eta,u) \mbox{ in } C([0,T];X_0), \mbox{ as } n\rightarrow\infty.
 \end{equation}
In particular, $$(\eta(0),u(0))=\lim_{n\rightarrow\infty}(\eta_{n}(0),u_n(0))=\lim_{n\rightarrow\infty}(\eta_{0,n},u_{0,n})=(\eta_0,u_0).$$ 
Consequently, passing to the weak limit, by Proposition \ref{existence},
we obtain $$(\eta(t),u(t))=S(t)(\eta_0,u_0).$$ Moreover, from \eqref{e3}, we obtain that
$$
\int_0^T \left(|\eta_{xx}(L,t)|^2+|\eta_{xx}(0,t)|^2\right)dt
\leq \liminf_{n\rightarrow \infty}\int_0^T \left(|\eta_{n,xx}(L,t)|^2+|\eta_{n,xx}(0,t)|^2\right)dt.
$$
Thus, we have that $(\eta,u)$ is the solution of the IBVP \eqref{homo1} with initial data $(\eta_0,w_0)$ which satisfies, additionally,
\begin{equation}\label{e6}
\eta_{xx}(L,t)=\eta_{xx}(0,t)=0
\end{equation}
and
\begin{equation}\label{e7}
\|(\eta_0,u_0)\|_{X_0}=1.
\end{equation}
Notice that \eqref{e7} implies that the solution $(\eta,u)$ can not be identically zero. However, from lemma bellow, one can conclude that $(\eta,u)= (0, 0)$, which drive us to a contradiction.
\end{proof}
\noindent\textbf{Step 2: {\em Reduction to a spectral problem}}
\vglue 0.2cm
\begin{lemma}\label{lem1}
For any $T > 0$, let $N_T$ denote the space of the initial states $(\eta_0,u_0) \in X_0$, such that the solution  $(\eta(t),u(t))=S(t)(\eta_0,u_0)$ of \eqref{homo1} satisfies \eqref{e6}. Then, $N_T=\{0\}$.
\end{lemma}
\begin{proof}
The proof uses the same arguments as those given in \cite[Theorem 3.7]{CaPaRo}. If $N_T\neq\{0\}$, the map  $(\eta_0,u_0) \in \C N_T \rightarrow A(N_T)\subset \C N_T$
(where $\C N_T$ denote the complexification of $N_T$) has (at least) one eigenvalue. Hence, there exists $\lambda \in \C$ and $\eta_0,u_0 \in  H^5(0,L)\setminus \{ 0 \}$, such that
\begin{equation*}
\begin{split}
\begin{cases}
\lambda\eta_0+ u'_0-au'''_{0}+bu'''''_{0}=0, & \text{in} \,\, (0,L),  \\
\lambda u_0 +\eta'_0-a\eta'''_{0} +b\eta'''''_{0}=0, & \text{in} \,\, (0,L),  \\
\eta_0(0)=\eta_0(L)=\eta'_0(0)=\eta'_0(L)=\eta''_0(0)=\eta''_0(L)=0, \\
u_0(0)=u_0(L)=u'_0(0)=u'_0(L)=u''_0(0)=u''_0(L)=0.
\end{cases}
\end{split}
\end{equation*}
To obtain the contradiction, it remains to prove that a triple $(\lambda,\eta_0,u_0)$ as above does not exist.
\end{proof}
\noindent\textbf{Step 3: {\em M\"obius transformation}}
\vglue 0.2cm
To simplify the notation, henceforth we denote $(\eta_0,u_0):=(\eta,u)$. Moreover, the notation $\{0,L\}$ means that the function is applied to $0$ and $L$, respectively.
\begin{lemma}\label{lem2}
Let $L>0$ and consider the assertion
\begin{equation*}
(\NN):\ \ \exists \lambda \in \C,  \exists (\eta,u) \in (H^2_0(0,L)\cap H^5(0,L))^2 \,\, \text{such that}\,\,
\end{equation*}
\begin{equation*}
\begin{cases}
\lambda\eta+ u'-au'''+bu'''''=0, & \text{in} \,\, (0,L),  \\
\lambda u +\eta'-a\eta''' +b\eta'''''=0, & \text{in} \,\, (0,L),  \\
\eta(x)=\eta'(x)=\eta''(x)=0, & \text{in} \,\, \{0,L\},\\
u(x)=u'(x)=u''(x)=0, & \text{in} \,\, \{0,L\}.
\end{cases}
\end{equation*}
Then, if $(\lambda,\eta,u) \in \C \times  (H^2_0(0,L)\cap H^5(0,L))^2$ is solution of $(\NN)$, then
$$\eta=u=0.$$
\end{lemma}
\begin{proof}
Firstly, let us consider the following change of variable $\varphi(x)=\eta(x)\pm u(x)$, thus we have the problem in only one equation:
\begin{equation}\label{e8}
\begin{cases}
\lambda\varphi+ \varphi'-a\varphi'''+b\varphi'''''=0, & \text{in} \,\, (0,L),  \\
\varphi(x)=\varphi'(x)=\varphi''(x)=0, & \text{in} \,\, \{0,L\}.
\end{cases}
\end{equation}
Note that, if we multiply the equation in \eqref{e8} by $\overline{\varphi}$ and integrate in $[0,L]$, it is easy to see that $\lambda$ is purely imaginary, i.e., $\lambda=ir$, for $r\in\mathbb{R}$. Now, we extend the function $\varphi$ to $\mathbb{R}$ by setting  $\varphi(x)=0$ for $x\not\in [0,L]$. The extended function satisfies
 \begin{equation*}
 \lambda\varphi+ \varphi'-a\varphi'''+b\varphi'''''=b\varphi''''(0)\delta_0^{'}-b\varphi''''(L)\delta_L^{'}+b\varphi'''(0)\delta_0-b\varphi'''(L)\delta_L,
\end{equation*}
in ${\mathcal S }'(\R )$,  where $\delta_\zeta$ denotes the Dirac measure at $x=\zeta$ and the derivatives $\varphi''''(0)$, $\varphi''''(L)$, $\varphi'''(0)$ and $\varphi'''(L)$ are those of the
function $\varphi$ when restricted to $[0,L]$.  Taking the Fourier transform of each term in the above system and integrating by parts, we obtain
\begin{equation*}
\begin{split}
\lambda \hat \varphi (\xi )  + i\xi \hat \varphi (\xi )  - a(i\xi )^3 \hat \varphi (\xi) +b  (i\xi )^5 \hat \varphi (\xi)  =& b(i\xi)\varphi'''(0)-b(i\xi)\varphi'''(L)e^{-iL\xi}\\&+b\varphi''''(0)-b\varphi''''(L)e^{-iL\xi}.
\end{split}
\end{equation*}
Setting $\lambda = -ir$ and $f_{\alpha}(\xi,L)=i \hat \varphi (\xi )$, from the equation above it follows that
\begin{equation*}
f_{\alpha}(\xi,L)=\frac{N_{\alpha}(\xi,L)}{q(\xi)},
\end{equation*}
with $N_{\alpha}(\cdot,L)$ defined by
\begin{equation}\label{Na}
N_{\alpha}(\xi,L)=\alpha_1i\xi-\alpha_2i\xi e^{-i\xi L}+\alpha_3-\alpha_4e^{-i\xi L}
\end{equation}
and
\begin{equation*}
q(\xi)=b\xi^5+a\xi^3+\xi+r,
\end{equation*}
where
$\alpha_i$, for $i=1,2,3,4$, are the traces of $b\varphi'''$ and $b\varphi''''$.

For each $r\in\mathbb{R}$ and $\alpha\in\mathbb{C}^4\setminus \{0\}$ let $\FF_{\alpha r}$ be the set of $L>0$ values, for which the function $f_{\alpha}(\cdot,L)$ is entire. 
We introduce the following statements, which are equivalent:
\begin{itemize}
\item[A1.] $f_{\alpha}(\cdot,L)$ is entire;
\item[A2.] all zeros, taking the respective multiplicities into account, of the polynomial $q$ are zeros of $N_{\alpha}(\cdot,L)$;
\item[A3.]  the maximal domain of $f_{\alpha}(\cdot,L)$ is $\mathbb{C}$.
\end{itemize}
To the function $f_{\alpha}(\cdot,L)$ to be entire, due to the equivalence between statement A1 and A2, we must have
\begin{equation*}
\frac{\alpha_1i\xi_i+\alpha_3}{\alpha_2i\xi_i+\alpha_4}=e^{-iL\xi_i},
\end{equation*}
where $\xi_i$ denotes the zeros of $q(\xi)$, for $i=1,2,3,4,5$.
Let us define, for $\alpha \in\mathbb{C}^4\setminus \{0\}$, the following discriminant  \begin{equation}\label{discri}
d(\alpha)=\alpha_1\alpha_3-\alpha_2\alpha_4.
\end{equation}
Then, for $\alpha\in\mathbb{C}^4\setminus \{0\}$, such that $d(\alpha)\neq0$ the M\"obius transformations can be introduced by
\begin{equation}\label{mobius}
M(\xi_i)=e^{-iL\xi_i},
\end{equation}
for each zero $\xi_i$ of the polynomial $q(\xi)$.

\vglue 0.2cm
\noindent The next claim analyzes the behavior of the roots of polynomial $q(\cdot)$:

\begin{claim}\label{roots}
The polynomial $q(\cdot)$ has exactly one real root, with multiplicity $1$ and two pairs of complex conjugate roots.
\end{claim}
\begin{proof}[Proof of the Claim \ref{roots}]
Initially, we suppose that $r\neq 0$. Note that the derivative of $q$ is given by
\begin{equation*}
q'(\xi)=5b\xi^4+3a\xi^2+1,
\end{equation*}
and its zeros are $\pm z_1$ and $\pm z_2$, where
\begin{equation*}
z_1=\sqrt{\frac{-3a-\sqrt{9a^2-20b}}{10b}} \quad\text{and}\quad z_2=\sqrt{\frac{-3a+\sqrt{9a^2-20b}}{10b}}.
\end{equation*}
It is easy to see that $z_1$ and $z_2$ belong to $\C \setminus \R$. Hence, the polynomial $q(\cdot)$ does not have critical points, which means that $q(\cdot)$ has exactly one real root. Suppose that $\xi_0 \in \R$ is the root of $q(\cdot)$ with multiplicity $m \leq 5$. Hence,
$$q(\xi_0)= q'(\xi_0)= ... = q^{(m-1)}(\xi_0)=0.$$
Consider the following cases:
\begin{enumerate}
\item[(i)] If $\xi_0$ has multiplicity $5$, it follows that $q(\xi_0)=0$ and $q''''(\xi_0)=120b\xi_0=0$,  implying that $\xi_0=0$ and $r=0$.
\item[(ii)] If $\xi_0$ has multiplicity $4$, it follows that $q'''(\xi_0)=60b\xi^2_0+6a=0$,  implying that $\xi_0 \in i\R$.
\item[(iii)] If $\xi_0$ has multiplicity $3$, it follows that $q(\xi_0)=0$ and $q''(\xi_0)=20b\xi^3_0+6a\xi_0=0$,  implying that $\xi_0=0$ and $r=0$ or $\xi_0 \in i\R$.
\item[(iv)] If $\xi_0$ has multiplicity $4$, it follows that $q'(\xi_0)=5b\xi^4_0+3a\xi_2+1=0$,  implying that $\xi_0 \in \C \setminus \R.$
\end{enumerate}
In any case, we have a contradiction, since $r\neq 0$ and $\xi_0 \in \R$.  Consequently, $q(\cdot)$ has exactly one real root, with multiplicity $1$. This means that this polynomial has two pairs of complex conjugate roots.

\vglue 0.2cm

Second, we suppose that $r=0$. Initially, note that from the derivation of the model (see the Appendix) we have that $4b>a^2$. Then, we obtain that $$ q(\xi)= \xi(b\xi^4+a\xi^2+1),$$ whose roots are $0, \pm   \rho$ and $\pm  k$ where
\begin{equation}
\rho ^2=  -\frac{a}{2b} +i\frac{\sqrt{4b-a^2}}{2b} \quad \text{and} \quad k=-\frac{a}{2b} -i\frac{\sqrt{4b-a^2}}{2b}= \overline{\rho^2}=\overline{\rho}^2.
\end{equation}
Thus, $q(\cdot)$ has two pairs of complex conjugate roots and one real root, proving Claim \ref{roots}.
\end{proof}
\vglue 0.2cm

Besides of the Claim \ref{roots} the following two auxiliary lemmas  are necessary to conclude the proof of the Lemma \ref{lem2}.  Their proofs can be found in \cite[Lemmas 2.1 and 2.2]{santos}, thus we will omit them.

\begin{lemma}\label{mobiuslemma2}
Let non null $\alpha \in   \C^4$ with $d(\alpha) = 0$ and $L > 0$ for $d(\alpha)$ defined in \eqref{discri}. Then, the set of the imaginary parts of the
zeros of $N_{\alpha}(\cdot,L)$ in \eqref{Na} has at most two elements.
\end{lemma}

\begin{lemma}\label{mobiuslemma}
For any $L > 0$, there is no M\"obius transformation $M$, such that
\begin{align*}
M(\xi)=e^{-iL\xi}, \quad \xi\in \{ \xi_1, \xi_2, \bar{\xi}_1, \bar{\xi}_2 \},
\end{align*}
with $\xi_1, \xi_2, \bar{\xi}_1, \bar{\xi}_2$ all distinct in $\C$.
\end{lemma}

Let us finish the proof of Lemma \ref{lem2}. To do this, we need to consider two cases: \begin{itemize}
\item[i.] $d(\alpha)\neq 0$;
\item[ii.] $d(\alpha)=0$,
\end{itemize}
where $d(\alpha)$ was defined in \eqref{discri}.

In fact, if $d(\alpha)\neq 0$, we can defined the M\"obius transformation. Let us assume, by contradiction, that there exists $L>0$  such that the function $f_a(\cdot,L)$ is entire. Then, all roots of the polynomial $q(\cdot)$ must satisfy \eqref{mobius}, i.e., there exists a M\"obius transformation that takes each root $\xi_0$ of $q(\cdot)$ into $e^{-iL\xi_0}$.
However, this contradicts Lemma \ref{mobiuslemma} and proves that if $(\NN)$ holds then $\FF_{\alpha r}=\emptyset$ for all $r\in\R$. On the other hand, suppose that $d(\alpha)=0$ and note that by using the claim \ref{roots}, we can conclude that the set of the imaginary parts of the polynomial $q(\cdot)$ has at least three elements, thus it follows from Lemma \ref{mobiuslemma2} that $\FF_{\alpha r}=\emptyset$ for all $r\in\R$. Note that in both cases, we have that $\FF_{\alpha r}=\emptyset$, which implies that $(\NN)$ only has the trivial solution for any $L>0$, and the proof of Lemma \ref{lem2} is archived.
\end{proof}

To close this section we derive an exponential stability result in each space $X_s$, for $s\in[0,5]$. To do this, for $s\in [0,5]$, let $X_s$ denote the collection of all the functions $$(\eta,u)\in[H^s_0 (0,L)]^2:=\{(\eta,u)\in [H^s(0,L)]^2:(\eta,u)^{(j)}(0)=0, (\eta,u)^{(j)}(L)=0\},$$ for $j=0,1,...,[s] $\footnote{For any real number $s$, $[s] $ stands for its  integer part.}, endowed with the Hilbertian norm $$\|(\eta,u)\|^2_{X_s}=\|\eta\|^2_{H^s(0,L)}+\|u\|^2_{H^s(0,L)}.$$ Using Theorem \ref{main_int} and some interpolation argument, we derive 
the following result.
\begin{corollary}
Let $\alpha_i$, $i=1,2$ be as in Theorem \ref{main_int}. Then, for any $s\in [0, 5]$, there exists a constant $C_s>0$, such that, for any $(\eta_0,u_0)\in X_s$, the solution $(\eta(t),u(t))$ of \eqref{homo1} belongs to $C(\mathbb{R}^+; X_s)$ and fulfills
\begin{equation}\label{exponential}
\|(\eta(t),u(t))\|_{X_s}\leq C_0e^{-\mu_0 t}\|(\eta_0,u_0)\|_{X_s}, \quad \forall t \geq 0.
\end{equation}
\end{corollary}
\begin{proof}
\eqref{exponential} was already been established for s = 0 in Theorem \ref{main_int}. Pick any $U_0=(\eta_0,u_0)\in X_5= D(A)$ and write $U(t)=(\eta(t),u(t))=S(t)U_0$. Let $V(t)=U_t(t) =AU(t)$. Then $V$ is the mild solution of the system
\begin{equation}\label{b1a}
\begin{split}
\begin{cases}
V_t=AV\\
V(0)=AU_0\in X_0,
\end{cases}
\end{split}
\end{equation}
hence, by using Theorem \ref{main_int}, estimate $\|V(t)\|_{X_0}\leq C_0e^{-\mu_0 t}\|V_0\|_{X_0}$, holds. Since $V(t) = AU(t)$, $V_0 = AU_0$, and the norms $\|U\|_{X_0}+ \|AU\|_{X_0}$  and $\|U\|_{X_5}$ are equivalent in $X_5$, we conclude that, for some constant $C_5 > 0$, we have that  $$\|U(t)\|_{X_5}\leq C_5e^{-\mu_0 t}\|U_0\|_{X_5}.$$  This proves \eqref{exponential} for $s = 5$. The fact that \eqref{exponential} is still valid for $0 < s < 5$ follows from a standard interpolation argument, since $X_s = [X_0, X_5]_{s/5}$.
\end{proof}

\section{Further comments and open problems}\label{Sec4}

In this section considerations will be done regarding the fifth order Boussinesq system \eqref{int1}-\eqref{int2}. It is important to note that the classical energy estimate does not provide any global (in time) a priori bounds for the solutions of the corresponding nonlinear model. Consequently, it does not lead to the existence of a global (in time) solution in the energy space. Due to the structure of the nonlinear terms, the same lack of a priori bounds also occurs when higher order Sobolev norms are considered (e. g. $H^s-$norm).  Because to this strict requirement, we cannot proceed as in \cite{CaPaRo,pazoto2008} and have only succeeded in deriving uniform decay
results for the linear system. However, for the full system, we can find solutions - in a certain sense - globally in time in $\overline{X}_s$.

\vspace{0.4cm}
\noindent$\bullet$\textit{ \textbf{Global well-posedness in time}}
\vspace{0.2cm}

Theorem \ref{main_int1} and Proposition \ref{extension} give us a  positive answer to the well-posedness problem. However, the following questions are still open:

\begin{question1}
Is the nonlinear system \eqref{int1}-\eqref{int2}, global well-posedness in time? If yes, should we expect some restriction on the initial data?
\end{question1}


\vspace{0.3cm}
\noindent$\bullet$\textit{ \textbf{One feedback on the boundary condition}}
\vspace{0.2cm}

If we consider in \eqref{int1}-\eqref{int2} with only one damping mechanism, that is, with $\alpha_1$ or $\alpha_2$ vanishing, we still have $E(t)$, defined by \eqref{energy}, decreasing along the trajectories of the linearized system associated to the model. Thus, the following question can be formulated:

\begin{question3}
Is still valid, with only one damping mechanism, the exponential stability for the linearized system associated to \eqref{int1}-\eqref{int2}?
\end{question3}

\vspace{0.4cm}
\noindent$\bullet$\textit{ \textbf{Exponential stability for the full system}}
\vspace{0.2cm}

Due the lack of the classical energy estimate for the nonlinear model we are not able to prove, by using, e.g., \cite{CaPaRo,pazoto2008}, the exponential stability for the full model \eqref{int1}-\eqref{int2}. Then, one natural question remains open:

\begin{question5}
Does the energy associated to the nonlinear system \eqref{int1}-\eqref{int2}, with one or two damping mechanism, converges to zero, as $t\rightarrow \infty$, for initial data on the energy space  $X_0$?
\end{question5}
\section{Appendix}\label{appendix}
The following fifth-order Boussinesq system
\begin{equation}
\label{b1''}
\begin{cases}
\eta_t + u_x-au_{xxx}+ a_1 (\eta u)_x +a_2(\eta u_{xx})_x + b u_{xxxxx}= 0,\\
u_t +\eta_x -a\eta_{xxx} +a_1uu_x+a_3(\eta\eta_{xx})_{x}+  a_4u_xu_{xx}    +b\eta_{xxxxx}=0, \\
\end{cases}
\end{equation}
with $a>0$, $b>0$, $a\neq b$, $a_1>0$, $a_2<0$, $a_3>0$ and $a_4>0$, 
can be derived from \eqref{fifthb} with a carefully choice of the parameters $\theta$, $\beta$ and $\tau$.

Indeed, taking $\tau= \frac{2}{3}-\theta^2,$ we have that $$\frac16\beta(3\theta^2-1)=\beta \left[ \frac12 (1-\theta^2) -\tau \right],$$ and $a=\frac16\beta(3\theta^2-1)$. On the other hand, taking $\theta^2=\frac12-\frac{1}{2\sqrt{5}}\approx 0,276393<\frac13$ and noting that $\beta >0$, it follows that
\begin{equation}\label{new3}
5\left(\theta^2-\frac15\right)^2=(\theta^2-1)(3-11\theta^2)
\end{equation}
and
\begin{equation*}
a=\frac12\beta\left(\theta^2-\frac13\right)<0.
\end{equation*}
Note that \eqref{new3} is equivalent to
\begin{align*}
5\left(\theta^2-\frac15\right)^2 = (\theta^2-1)\left( \frac93-11\theta^2\right)&= (\theta^2-1)\left( \theta^2-5+12\tau\right) \\
&= (\theta^2-5)(\theta^2-1)+12(\theta^2-1)\tau.
\end{align*}
Thus,
\begin{align*}
&\frac{5}{24}\left(\theta^2-\frac15\right)^2= \frac{1}{24}(\theta^2-5)(\theta^2-1)+\frac12(\theta^2-1)\tau  \\
&\Leftrightarrow \quad \frac{1}{120}\left(25\theta^4-10\theta^2+1\right)= \frac{1}{24}(\theta^2-6\theta^2+5)+\frac12(\theta^2-1)\tau,
\end{align*}
and $b$ can define as $$b=\frac{\beta^2}{120}\left(25\theta^4-10\theta^2+1\right)= \beta^2\left[\frac{1}{24}(\theta^2-6\theta^2+5)+\frac12(\theta^2-1)\tau \right]>0.$$
Finally, with the choice of
\begin{multline*}
a_1=\alpha >0, \quad a_2=\frac12\alpha\beta(\theta^2-1)<0,\quad a_3=\alpha\beta>0\quad \text{and}\quad a_4=\alpha\beta(2-\theta^2)>0,
\end{multline*}
we obtain \eqref{b1''}.

\subsection*{Acknowledgments}
R. A. Capistrano--Filho was partially supported by CNPq (Brazil) by the grants 306475/2017-0 and Propesq (UFPE) by Edital  ``Qualis A''. A. F. Pazoto was partially supported by CNPq (Brazil). This work was carried out during two visits of the first author to the Federal University of Rio de Janeiro and one visit of the second author to the Federal University of Pernambuco. The authors would like to thank both Universities for its hospitality.


\begin{thebibliography}{}


\bibitem{adams}
R. Adams, Sobolev spaces, Acad. Press, New York, San Francisco, London, (1975).


\bibitem{bardos}
C. Bardos, G. Lebeau and J. Rauch, Sharp sufficient conditions for the observation, control, and stabilization of waves from the boundary, SIAM J. Control Optim. (30) (1992) 1024--1065.

\bibitem{bergh}
J. Bergh, J. L\"ofstr\"om, Interpolation Spaces. An Introduction, Springer-Verlag.

\bibitem{bona2002}
J. L. Bona, M. Chen and J. C. Saut, Boussinesq equations and other systems for small-amplitude long waves in nonlinear dispersive media I: Derivation and linear theory, Journal of Nonlinear Science. (12) (2002) 283--318.

\bibitem{bona2004}J. L. Bona, M. Chen and J. C. Saut, Boussinesq equations and other systems for small-amplitude long waves in nonlinear dispersive media: II. The nonlinear theory, Nonlinearity. (17) (2004)  925--952.

\bibitem{BonaSunZhang2003}J. L.  Bona, S. M. Sun and B. Y. Zhang, A nonhomogeneous boundary-value problem for the Korteweg–de Vries equation posed on a finite domain, Comm. Partial Differential Equations. (28) (2003) 1391--1436.

\bibitem{bona1976}
J. L. Bona and R. Smith, A model for the two-way propagation of water waves in a channel, Math. Proc. Camb. Phil. Soc. (79) (1976) 167--182.

\bibitem {boussinesq1}J. V. Boussinesq,  Th\'{e}orie g\'{e}n\'{e}rale
des mouvements qui sont propag\'{e}s dans un canal rectangulaire horizontal, C. R. Acad. Sci. Paris.  (72) (1871)  755--759.

\bibitem{CaGa}
R. A. Capistrano-Filho and F.A. Gallego, Asymptotic behavior of Boussinesq system of KdV-KdV type,  J. Differential Equations. (265) (2018) 2341--2374.

\bibitem{CaPaRo}
R. A. Capistrano-Filho, A. F. Pazoto and L. Rosier, Control of Boussinesq system of KdV-KdV type on a bounded interval,  ESAIM Control Optimization and Calculus Variations doi.org/10.1051/cocv/2018036 (2018).



\bibitem{CoSaut1988} P. Constantin and J.C. Saut, Local smoothing properties of dispersive equations, J. Amer. Math. Soc. (1) (1988) 413--446.

\bibitem{daripa2003}
P. Daripa and R. K. Dash, A class of model equations for bi-directional propagation of capillary–gravity waves, International journal of engineering science (41)  (2003) 201--218.

\bibitem{glassguerrero}
O. Glass and S. Guerrero, On the controllability of the fifth order Korteweg-de Vries equation, Annales de l'Institut Henri Poincaré Analyse Non Linéaire. (26) (2009) 2181-2209.



\bibitem{olver1984}
P. J. Olver, Hamiltonian and non-Hamiltonian models for water waves, Lecture Notes in Physics, vol.195, Springer Verlag,  273–-296, (1984).

\bibitem{KenPonVeg1991}
C. E. Kenig, G. Ponce and L. Vega, Oscillatory integrals and regularity of dispersive equations, Indiana Univ. Math. J. (40) (1991) 33--69.


\bibitem{lions}
J.-L. Lions and E. Magenes,  Non-homogeneous boundary value problems and applications (Vol. 1). Springer Science and Business Media, (2012).




\bibitem{pazy1983} A. Pazy, Semigroups of Linear Operators and Applications to Partial Differential Equations. Berlin--Heidelberg--New York--Tokyo, Springer--Verlag (1983).

\bibitem{pazoto2008}
A. F. Pazoto and L.  Rosier, Stabilization of a Boussinesq system of KdV-KdV type, Systems and Control Letters. (57)  (2008)  595--601.

\bibitem{Rosier}
L. Rosier, Exact boundary controllability for the Korteweg–de Vries equation on a bounded domain, ESAIM Control Optim. Calc. Var. (2) (1997) 33--55.

\bibitem{santos} A. L. C. dos Santos, P. N. da Silva and C. F. Vasconcellos, Entire functions related to stationary solutions of the Kawahara equation, Electron. J. Differential Equations. (43) (2016) 13pp.

\bibitem{SauXu2012}J.-C. Saut and Li Xu, The Cauchy problem on large time for surface waves Boussinesq systems, J. Math. Pures Appl. (9) (2012) 635--662.

\bibitem{SauXu2012a} J.-C. Saut and Li Xu, Well-posedness on large time for a modified full dispersion system of surface waves, J. Math. Phys. (53) (2012) 1--22.

\bibitem{SauWanXu2017} J.-C. Saut, Chao Wang and Li Xu, The Cauchy problem on Large Time for Surface-Waves-Type Boussinesq Systems II,  SIAM J. Math. Anal. (4) (2017) 2321--2386.

\bibitem{ZhaoZhang2014} X. Zhao and B.-Y. Zhang, Boundary smoothing properties of the Kawahara equation posed on the finite domain, J. Math. Anal. Appl. (417) (2014) 519--536.


\end{thebibliography}
\end{document}